\documentclass[reqno]{amsart}

\usepackage[a4paper]{geometry}
\usepackage{latexsym,amssymb,amsmath,bm,graphicx,enumitem,verbatim}
\usepackage{tikz}
\usepackage{todonotes,marvosym}
\usepackage{float}

\usepackage[urlcolor=blue]{hyperref}
\usepackage{enumitem}
\usepackage[all]{xy}

\usepackage[scr=boondox]{mathalfa}

\usetikzlibrary{positioning,decorations.pathmorphing,fit,petri,backgrounds,decorations.pathreplacing}

\makeatletter
\@namedef{subjclassname@2020}{\textup{2020} Mathematics Subject Classification}
\makeatother
\title{Outward compactness}

\DeclareMathOperator{\Lim}{Lim}

\DeclareMathOperator{\Ord}{Ord}

\newcommand{\Pow}{P}

\DeclareMathOperator{\id}{id}

\newcommand{\LL}{\mathcal{L}}
\newcommand{\TT}{\mathcal{T}}
\newcommand{\UU}{\mathcal{U}}

\DeclareMathOperator{\ot}{ot}

\DeclareMathOperator{\crit}{crit}

\DeclareMathOperator{\rank}{rank}

\DeclareMathOperator{\ZFC}{ZFC}

\newcommand{\Coll}{\mathrm{Col}}
\newcommand{\Set}[2]{\{#1~\vert~#2\}}

\newcommand{\pow}{\mathcal{P}}

\newcommand{\anf}[1]{{\text{``}\hspace{0.3ex}{#1}\hspace{0.01ex}\text{''}}}

\newcommand{\Str}{Str}

\newcommand{\map}[3]{{#1}:{#2}\longrightarrow{#3}}
\newcommand{\Map}[5]{{#1}:{#2}\longrightarrow{#3};~{#4}\longmapsto{#5}}
\newcommand{\seq}[2]{\langle{#1}~\vert~{#2}\rangle}
\newcommand{\ooo}{\mathscr{o}}
\newcommand{\up}{{{}_\uparrow}}

\definecolor{dblue}{rgb}{0,0,0.70}
\hypersetup{
	unicode=true,
	colorlinks=true,
	citecolor=dblue,
	linkcolor=dblue,
	anchorcolor=dblue
}

\author{Peter Holy}
\address{Peter Holy. Institut f\"ur Diskrete Mathematik und Geometrie, TU Wien, Wiedner Hauptstra{\ss}e 8--10/104, 1040 Wien, Austria}
\email{peter.holy@tuwien.ac.at}

\author{Philipp L\"{u}cke}
\address{Philipp L\"{u}cke. Fachbereich Mathematik, Universit\"at Hamburg, Bundesstra{\ss}e 55, Hamburg, 20146, Germany}
\email{philipp.luecke@uni-hamburg.de}

\author{Sandra M\"uller}
\address{Sandra M\"uller. Institut f\"ur Diskrete Mathematik und Geometrie, TU Wien, Wiedner Hauptstra{\ss}e 8--10/104, 1040 Wien, Austria}
\email{sandra.mueller@tuwien.ac.at}


\thanks{The authors would like to thank Menachem Magidor for very helpful discussions on some of the topics of this paper. 
 In addition, the authors would like to thank the anonymous referee for numerous helpful suggestions and corrections.
The first and third authors gratefully acknowledge funding
from Austrian Science Fund (FWF). This research was funded in whole or in part by the Austrian Science Fund (FWF) [10.55776/V844, 10.55776/Y1498, 10.55776/I6087]. For open access purposes, the authors have applied a CC BY public copyright license to any author accepted manuscript version arising from this submission.
 The second author gratefully acknowledges supported by the Deutsche Forschungsgemeinschaft (Project number 522490605) and the Spanish Government under grant EUR2022-134032.
}

\begin{document}

\begin{abstract}
  We introduce and study a new type of compactness principle for strong logics that, roughly speaking, infers the satisfiability of a theory from the satisfiability of its small subtheories in certain outer models of the set-theoretic universe. We refer to this type of compactness property as \emph{outward compactness}, and we show that instances of this type of principle for second-order logic can be used to characterize various large cardinal notions between measurability and extendibility, directly generalizing a classical result of Magidor that characterizes extendible cardinals as the strong compactness cardinals of second-order logic. 
 In addition, we generalize a result of Makowsky that shows that Vop\v{e}nka's Principle is equivalent to the existence of compactness cardinals for all abstract logics by characterizing the principle \anf{$\Ord$ is Woodin} through outward compactness properties of abstract logics.  
  \end{abstract}

\maketitle


\newcommand{\chkfront}{\scalebox{1.6}[1.0]{$\vee$}}
\newcommand{\GG}{\mathcal G}
\newcommand{\PP}{\mathcal P}
\newcommand{\Hyp}{\mathrm{Hyp}}

\newtheorem{fact}{Fact}[section]
\newtheorem{lemma}[fact]{Lemma}
\newtheorem{theorem}[fact]{Theorem}
\newtheorem{corollary}[fact]{Corollary}
\newtheorem{claim}{Claim}
\newtheorem*{claim*}{Claim}
\newtheorem{subclaim}[fact]{Subclaim}
\newtheorem{conjecture}[fact]{Conjecture}
\newtheorem{observation}[fact]{Observation}
\newtheorem{proposition}[fact]{Proposition}
\newtheorem{counterexample}[fact]{Counterexample}
\theoremstyle{definition}
\newtheorem{question}[fact]{Question}
\newtheorem{remark}[fact]{Remark}
\newtheorem{definition}[fact]{Definition}
\newtheorem{example}[fact]{Example}

\newcommand{\thistheoremname}{}
\newtheorem*{genericthm}{\thistheoremname}
\newenvironment{namedthm}[1]
  {\renewcommand{\thistheoremname}{#1}%
   \begin{genericthm}}
  {\end{genericthm}}


\section{Introduction}\label{section:introduction}

The work presented in this paper contributes to the study of the deep connections between large cardinals and strong logics. Its starting point is a classical result of Magidor \cite{MR0295904} that relates the existence of extendible cardinals to the compactness properties of second-order logic~$\LL^2$. 
Remember that an infinite cardinal $\kappa$ is \emph{extendible} if for every ordinal $\eta>\kappa$, there is an ordinal~$\zeta$ and a non-trivial elementary embedding $\map{j}{V_\eta}{V_\zeta}$ satisfying $\crit(j)=\kappa$ and $j(\kappa)>\eta$. 
Moreover, given a cardinal $\kappa$, an $\LL^2$-theory $T$ is \emph{${<}\kappa$-satisfiable} if every subtheory of $T$ of cardinality less than $\kappa$ is satisfiable. Finally, an infinite cardinal $\kappa$ is a \emph{strong compactness cardinal for $\LL^2$} if every ${<}\kappa$-satisfiable $\LL^2$-theory is satisfiable. The next theorem summarizes Magidor's characterization of extendible cardinals through compactness properties of $\LL^2$:

\begin{theorem}[Magidor \cite{MR0295904}]\label{theorem:Magidor}
 A cardinal $\kappa$ is  a strong compactness cardinal for $\LL^2$ if and only if there is an extendible cardinal less than or equal to $\kappa$.  
\end{theorem}

 The results of this paper are motivated by the aim to obtain analogous characterizations for various well-studied large cardinal properties below extendibility, {e.g.,}  measurability, strongness and supercompactness.\footnote{Analogous characterizations for large cardinal properties below measurability can be found in \cite{bdgm}, \cite{FuSa22} and \cite{SubtleOrd}.} 
 More precisely, for a given large cardinal property $\varphi$, we want to assign, uniformly in a parameter~$\kappa$, natural and rich classes 
 $\TT_\kappa$ of ${<}\kappa$-satisfiable $\LL^2$-theories to infinite cardinals $\kappa$ in a way that  ensures that for all such $\kappa$, all theories in $\TT_\kappa$ are satisfiable if and only some $\lambda\le\kappa$ satisfies $\varphi(\lambda)$.\footnote{Characterizations of large cardinals through strong logics of a somewhat different flavor can be found in \cite{MR4093885} and  \cite{bdgm}.}   
 We start by presenting a characterization of measurability through compactness properties of~$\LL^2$ that  illustrates the main concepts on which our more general results are based.\footnote{Note that, using similar arguments, folklore results already provide a characterization  of measurability through strong logics, by showing that measurability corresponds to  chain compactness for the logic $\LL(Q^{WF})$, the extension of first-order logic by the well-foundedness quanitifier     (see {\cite[Section 2]{bdgm}} and {\cite[Exercise 4.2.6]{MR409165}}, where  stronger infinitary logics are used).}

 As a first step, we show that for every uncountable cardinal $\kappa$, there is a canonical ${<}\kappa$-satisfiable $\LL^2$-theory $T_\kappa$ whose satisfiability is equivalent to the existence of a measurable cardinal less than or equal to $\kappa$. 
    To obtain this theory, first recall that there is an $\LL^2$-sentence in the language of set theory that holds in a model   if and only if its $\in$-relation is well-founded. Next,  work in the language that extends the language of set theory by constant symbols  $c_x$ for every $x\in V_{\kappa+1}$ and constant symbols $d_\gamma$ for all $\gamma\leq\kappa$, and let $T_\kappa$ be the $\LL^2$-theory consisting of the union of the following theories:
  \begin{itemize}
    \item The elementary first-order diagram of $V_{\kappa+1}$ in the language of set theory, making use of the constant symbols $c_x$. 

    \item All (first-order) sentences of the form $d_\beta\in d_\gamma\in c_\kappa$ for $\beta<\gamma\leq\kappa$. 
    
    \item The (second-order) sentence stating that the $\in$-relation is well-founded. 
  \end{itemize}
 The theory $T_\kappa$ is obviously ${<}\kappa$-satisfiable, since for every subtheory $F$ of $T_\kappa$ of cardinality less than $\kappa$, we can easily find a model of $F$ with underlying set $V_{\kappa+1}$. 
 Moreover, if there is a measurable cardinal less than or equal to $\kappa$, then $T_\kappa$ is satisfiable because this assumption allows us to use 
  standard iteration arguments (see {\cite[Corollary 19.7(b)]{kanamori}}) to find a transitive class $M$ and a non-trivial elementary embedding $\map{j}{V}{M}$ with $j(\crit(j))>\kappa$, and this allows us to construct a model $N$ of $T_\kappa$ with underlying set $j(V_{\kappa+1})$, $c_x^N=j(x)$ for all $x\in V_{\kappa+1}$ and $d_\gamma^N=\gamma$ for all $\gamma\leq\kappa$. 
 Finally, if $T_\kappa$ is satisfiable, then there exists a transitive set $N$ and an elementary embedding $\map{j}{V_{\kappa+1}}{N}$ with $j(\kappa)>\kappa$, and the existence of such an embedding directly implies that some cardinal less than or equal to $\kappa$ is measurable.

 Following the approach outlined above, we now want to isolate a natural satisfiability property of a theory $T$ with respect to infinite cardinals $\kappa$, that strengthens ${<}\kappa$-satisfiability, is possessed by $T_\kappa$ at every infinite cardinal $\kappa$, and implies the satisfiability of $T$ at cardinals $\kappa$ that are greater than or equal to a measurable cardinal. 
 Our definition of this property is motivated by the observation that for every cardinal $\lambda<\kappa$, 
 the theory $T_\kappa$ is not only ${<}\lambda$-satisfiable in our ground model $V$, but  remains ${<}\lambda$-satisfiable when we pass to outer models\footnote{{I.e.,} models of $\ZFC$ in which $V$ is a transitive class containing all ordinals.}  in which $\lambda$ is still a cardinal, as witnessed by the structure $V_{\kappa+1}^V$,\footnote{Whenever $M$ is an inner or outer model of our set-theoretic universe $V$ and $\alpha$ is an ordinal, we write $V_\alpha^M=\Set{x\in M}{\rank(x)<\alpha}$. In particular, when working in some outer model of $V$, we use the notation $V_\alpha^V$ for initial segments of (the ground model) $V$.} by the absoluteness of well-foundedness. 
 In addition, if we na\"ively assumed that the property that an $\LL^2$-theory $T$ is ${<}\lambda$-satisfiable in every outer model in which $\lambda$ is a cardinal could be uniformly expressed by a first-order formula with parameters $T$ and $\lambda$, it would follow that for every cardinal $\kappa$ greater than or equal to a measurable cardinal, every $\LL^2$-theory $T$ with this property  is in fact satisfiable in $V$. 
 We would argue as follows: By our assumptions, there is a cardinal $\lambda$ with $T\in H_\lambda$ and an elementary embedding $\map{j}{V}{M}$ with $\crit(j)\leq\kappa$ and $j(\kappa)>\lambda$. 
Using our na\"ive assumption, we apply the elementarity of $j$ to conclude that the $\LL^2$-theory $j(T)$ is ${<}\lambda$-satisfiable in every outer model of $M$ in which $\lambda$ is still a cardinal. 
In particular, since $V$ is such an outer model of $M$, it follows that $j(T)$ is a ${<}\lambda$-satisfiable $\LL^2$-theory in $V$, and this implies that the pointwise image $j[T]\subseteq j(T)$ of $T$ under $j$, which is of size less than $\lambda$, is a satisfiable $\LL^2$-theory in $V$. 
But this conclusion also shows that $T$ is satisfiable in $V$, because the finitary character of $\LL^2$-formulae allows us to identify $T$ and $j[T]$  via the renaming of the symbols of the given language induced by $j$.

 Since the above assumption on the uniform first-order definability of statements about truth in outer models is easily seen to be too na\"ive, we will introduce several concepts that allow us to turn the above approach into a mathematically sound  argument.

\begin{definition}
  Given a subtheory $F$ of $\ZFC$ and a transitive set $M$, a transitive set $N\supseteq M$ is an \emph{outer $F$-model of $M$} if $F$ holds in $N$ and the sets $M$ and $N$ have the same ordinals. 
\end{definition}

 In the following, we let $\ZFC^*$ denote the subtheory of $\ZFC$ that contains all the single axioms of $\ZFC$, together with the replacement and separation schemes  restricted to $\Sigma_2$-formulae. 
 Note that this theory proves the \emph{$\Sigma_2$-Recursion Theorem} and therefore proves that all levels of the \emph{von Neumann hierarchy} $\seq{V_\alpha}{\alpha\in\Ord}$ are sets. 
 Moreover, it is strong enough to yield the existence of the second-order satisfaction relation $\models_{\LL^2}$ for set-sized models,\footnote{This is because the second-order satisfaction relation  can be defined by a recursion over subformulae that is based on a $\Sigma_2$-function: we proceed as for the first-order satisfaction relation, however we have to essentially assert that $\exists W ~ W=\mathcal{P}(M)$ holds for each set-sized model with underlying set $M$, which is already a $\Sigma_2$-statement, and then we can bound all other relevant quantifiers by $W$. To prove that such a form of recursion can successfully be performed only requires replacement for $\Sigma_2$-formulae (a version of this argument for $\Sigma_1$-recursion can be found in \cite[Theorem 6.4]{barwise}).} 
  in a way that for every second-order formula $\varphi(v_0,\ldots,v_{m-1},W_0,\ldots,W_{n-1})$ with first-order variables $v_0,\ldots,v_{m-1}$ and second-order variables $W_0,\ldots,W_{n-1}$,\footnote{In the following, we will usually use uppercase letters for second-order variables and lowercase letters for first-order variables in $\LL^2$-formulas.}
  $\ZFC^*$ proves that for every non-empty set $M$, all $x_0,\ldots,x_{m-1}\in M$ and all $Y_0,\ldots,Y_{n-1}\in\mathcal P(M)$, the statement $$\langle M,\in\rangle\models_{\LL^2}\varphi(x_0\ldots,x_{m-1},Y_0,\ldots,Y_{n-1})$$ holds if and only if the first order formula $$\varphi^M(x_0\ldots,x_{m-1},Y_0,\ldots,Y_{n-1})$$ holds, where $\varphi^M$ denotes the  \emph{first-order relativization}  of $\varphi$ to $M$.\footnote{Inductively defined by   $(v_0\in v_1)^M\equiv v_0\in v_1$, 
  $(v\in W)^M\equiv v\in W$, 
  $(\neg\varphi)^M\equiv\neg\varphi^M$, $(\varphi\wedge\psi)^M\equiv\varphi^M\wedge\psi^M$, $(\forall x ~ \varphi(x))^M\equiv\forall x\in M ~ \varphi^M(x)$ and $(\forall X ~ \varphi(X))^M\equiv\forall x\subseteq M ~ \varphi^M(x)$. Note that, violating our above convention, we consider $W$ to be a first order variable after relativizing $v\in W$ to $M$ (alternatively, we could replace $W$ by a new and otherwise unused variable symbol $w$).} 
 In particular, the theory $\ZFC^*$ allows us to  uniformly speak about the satisfiability of $\LL^2$-theories.

\begin{definition}\label{definition:outward compactnessSimple}
 \begin{enumerate}
  \item Given an infinite cardinal $\kappa$, an $\LL^2$-theory $T$ is \emph{${<}\kappa$-outward satisfiable} if  for all cardinals $\lambda<\kappa$ and all 
  sufficiently large cardinals $\vartheta>\kappa$ with $T\in V_\vartheta$, 
  the partial order $\Coll(\omega,\vartheta)$ forces that \anf{\emph{$T$ is ${<}\lambda$-satisfiable in every outer $\ZFC^*$-model of $V_\vartheta^V$ in which $\lambda$ is  a cardinal}}. 
  
    \item A cardinal $\kappa$ is an \emph{outward compactness cardinal for $\LL^2$} if all ${<}\kappa$-outward satisfiable $\LL^2$-theories are satisfiable. 
 \end{enumerate}
\end{definition}

 Note that the only properties of the partial order $\Coll(\omega,\vartheta)$ that will be relevant in this paper are that it forces $\vartheta$ to be countable, and its definition is absolute between $V$ and all of its inner models. 
  Moreover,  note that in any $\Coll(\omega,\vartheta)$-generic extension $V[G]$, any  outer $\ZFC^*$-model~$N$ of $V_\vartheta^V$  is countable in $V[G]$: $N$ provides any of its levels $V_\alpha^N$ with some cardinality less than~$\vartheta$, and since $\vartheta$ is countable in $V[G]$, it follows that $N$ is a countable union of countable sets, and thus itself countable. 
Before we state our characterization of measurable cardinals, we make some easy observations that relate   outward compactness cardinals to compactness cardinals.

\begin{proposition}\label{proposition:BasicProp}
 \begin{enumerate}
  \item\label{item:UpOutwards} If $\kappa$ is an outward compactness cardinal for $\LL^2$, then every cardinal greater than $\kappa$ is an outward compactness cardinal for $\LL^2$. 
  
  \item\label{item:OutwardConsBoundedCons}   Given a limit cardinal $\kappa$, every ${<}\kappa$-outward satisfiable $\LL^2$-theory is ${<}\kappa$-satisfiable.  
  
  \item\label{item:StrongCompactOutward} Every strong compactness cardinal for $\LL^2$ is an outward compactness cardinal for $\LL^2$. 
 \end{enumerate}
\end{proposition}

\begin{proof}
 \eqref{item:UpOutwards} Let $\rho>\kappa$ be a cardinal and let $T$ be a ${<}\rho$-outward satisfiable $\LL^2$-theory.  
 Let   $\lambda<\kappa$ be a cardinal, let $\vartheta>\rho$ be a sufficiently large cardinal  with $T\in V_\vartheta$, let $G$ be $\Coll(\omega,\vartheta)$-generic over $V$, and let $N$ be an outer $\ZFC^*$-model of $V_\vartheta^V$ in $V[G]$ in which $\lambda$ is  a cardinal. Then, $T$ is ${<}\lambda$-satisfiable in $N$. 
 This shows that $T$ is ${<}\kappa$-outward satisfiable, and hence $T$ is satisfiable.

 \eqref{item:OutwardConsBoundedCons}  Let $T$ be an $\LL^2$-theory that is ${<}\kappa$-outward satisfiable, and assume, for the sake of a contradiction, that it is not ${<}\kappa$-satisfiable. Fix an unsatisfiable subtheory $T_0$ of $T$ of cardinality less than $\kappa$ and a cardinal $\lambda<\kappa$ with $\vert T_0\vert<\lambda$. 
 Pick a cardinal $\vartheta>\kappa$ such that $V_\vartheta$ is a model of $\ZFC^*$ and $T$ is an element of $V_\vartheta$. Let $G$ be $\Coll(\omega,\vartheta)$-generic over $V$. Then, in $V[G]$, the set $V_\vartheta^V$ is an outer $\ZFC^*$-model of itself in which $\lambda$ is a cardinal and hence $T$ is ${<}\lambda$-satisfiable in $V_\vartheta^V$. This shows that $T_0$ is satisfiable in $V_\vartheta^V$ and, by the nature of the satisfaction relation of $\LL^2$, we can conclude that $T_0$ is satisfiable in $V$, contradicting our assumption.

 \eqref{item:StrongCompactOutward} Let $\rho$ be a strong compactness cardinal for $\LL^2$. By Theorem \ref{theorem:Magidor}, there exists an extendible cardinal $\kappa\leq\rho$. In this situation, Theorem \ref{theorem:Magidor} together with \eqref{item:OutwardConsBoundedCons} implies that $\kappa$ is an outward compactness cardinal for $\LL^2$, and we can apply \eqref{item:UpOutwards} to conclude that $\rho$ also has this property. 
\end{proof}

 We now present the characterization of measurable cardinals that realizes the approach outlined above. The arguments appearing in its proof already contain many of the key ideas utilized in the later sections of this paper. Some parts of these arguments already appeared in a slightly different form in the na\"ive approach outlined above, but will be repeated within the proof below for the sake of clarity of presentation.

\begin{theorem}\label{theorem:measurable}
 A cardinal  $\kappa$ is an outward compactness cardinal for $\LL^2$ if and only if there exists a measurable cardinal less than or equal to $\kappa$.  
\end{theorem}

\begin{proof} 
 First, assume that  $\kappa$ is a measurable cardinal and  $T$ is a ${<}\kappa$-outward satisfiable $\LL^2$-theory. Pick a cardinal $\lambda>\kappa$ with $T\in H_\lambda$. 
  Using standard iteration arguments (see {\cite[Corollary 19.7(b)]{kanamori}}), we  find an inner model $M$ and an elementary embedding $\map{j}{V}{M}$ such that  $\crit(j)=\kappa$ and $j(\kappa)>\lambda$. 
 Pick a sufficiently large cardinal $\vartheta>j(\kappa)$ such that $j(\vartheta)=\vartheta$ and~$V_\vartheta$ is a model of $\ZFC^*$.  
  Elementarity ensures that in $M$, the $\LL^2$-theory $j(T)$ is ${<}j(\kappa)$-outward satisfiable. 
  Let $G$ be $\Coll(\omega,\vartheta)$-generic over $V$. Since $V_\vartheta^M$ is a model of $\ZFC^*$, the observation after Definition \ref{definition:outward compactnessSimple} shows that $V_\vartheta^M$ is countable in $M[G]$. 
 Our setup ensures that, in $M[G]$, the $\LL^2$-theory $j(T)$ is ${<}\lambda$-satisfiable in every countable outer $\ZFC^*$-model of $V_\vartheta^M$ in which $\lambda$ is  a cardinal. 
  This statement can be formulated by a $\Pi_1$-formula using parameters contained in $H_{\aleph_1}^{M[G]}$ and is therefore provably equivalent to a $\mathbf{\Pi}^1_2$-statement whose parameters are real numbers  coding the original parameters (see {\cite[Lemma 25.25]{MR1940513}}). 
  Therefore, \emph{Shoenfield absoluteness} implies that the given statement also holds in $V[G]$. 
  But, in $V[G]$, the set $V_\vartheta^V$ is a countable outer $\ZFC^*$-model of $V_\vartheta^M$ in which $\lambda$ is  a cardinal, and hence $j(T)$ is ${<}\lambda$-satisfiable in $V_\vartheta^V$.  
   Now, note that $j[T]\subseteq j(T)$ is an element of $V_\vartheta^V$ and has cardinality less than $\lambda$ in $V_\vartheta^V$. 
      Therefore, we can conclude that $j[T]$ is satisfiable in $V_\vartheta^V$, and this implies that $j[T]$ is a satisfiable $\LL^2$-theory in $V$. 
   But this also shows that $T$ is satisfiable in $V$, because the finitary character of $\LL^2$-formulae ensures that we can identify $T$ and $j[T]$ via the  renaming of symbols induced by $j$. These computations show that $\kappa$ is an outward compactness cardinal for $\LL^2$ and we can now apply Proposition \ref{proposition:BasicProp}.\eqref{item:UpOutwards} to see that every cardinal greater than $\kappa$ is also an outward compactness cardinal for $\LL^2$.

 Next, assume that $\kappa$ is an outward compactness cardinal for $\LL^2$. Consider the  $\LL^2$-theory $T_\kappa$ defined earlier, {i.e.,} we extend the language of set theory by constant symbols  $c_x$ for every $x\in V_{\kappa+1}$ and constant symbols $d_\gamma$ for all $\gamma\leq\kappa$, and define $T_\kappa$ to consist of: 
  \begin{itemize}
    \item The elementary first-order diagram of $V_{\kappa+1}$, making use of the constant symbols $c_x$. 
    
    \item All (first-order) sentences of the form $d_\beta\in d_\gamma\in c_\kappa$ for $\beta<\gamma\leq\kappa$. 
    
    \item The (second-order) sentence stating that the $\in$-relation is well-founded. 
  \end{itemize}

  \begin{claim*}
    The $\LL^2$-theory $T_\kappa$ is ${<}\kappa$-outward satisfiable. 
  \end{claim*}
  
  \begin{proof}[Proof of the Claim]
   Let $\lambda<\kappa$ be a cardinal, let $\vartheta>\kappa$ be a cardinal  with  $T_\kappa\in V_\vartheta$, let $G$ be  $\Coll(\omega,\vartheta)$-generic over $V$ and let  $N\in V[G]$ be an outer $\ZFC^*$-model  of $V_\vartheta^V$ in which $\lambda$ is a cardinal. Then $V_{\kappa+1}^V$ is an element of $N$.  
   Now, work in $N$ and fix a subtheory    $T_0$ of $T_\kappa$ of cardinality less than $\lambda$. 
  Since $\lambda$ is a cardinal, the axioms of $\ZFC^*$ allow us to construct a model of $T_0$ with domain $V_{\kappa+1}^V$ that interprets all constant symbols $c_x$ appearing in $T_0$ by the corresponding element $x$,  and interprets the (less than $\lambda$-many) constant symbols $d_\beta$ appearing in  $T_0$ as suitable  elements of $\lambda$.  
  \end{proof}

  Using the fact  that $\kappa$ is an outward compactness cardinal for $\LL^2$, we thus obtain that  $T_\kappa$ is satisfiable. 
  The definition of $T_\kappa$ now ensures that $T_\kappa$ has a transitive model $M$ and, by sending $x\in V_{\kappa+1}$ to $c_x^M$,   we obtain a non-trivial  elementary embedding $\map{j}{V_{\kappa+1}}{M}$ with  $j(\kappa)>d_\kappa^M\geq \kappa$. 
 As above, the existence of such an embedding implies that there is a measurable cardinal less than or equal to $\kappa$.   
\end{proof}

\begin{remark}
 \begin{enumerate}
  \item  In \cite[Theorem 4]{MR0295904}, Magidor proves that a cardinal $\kappa$ is extendible if and only if $\kappa$ is a strong compactness cardinal for the logic $\LL^2_{\kappa,\omega}$, that extends $\LL^2$ by allowing conjunctions and disjunctions of length less than $\kappa$. 
 An easy adaptation of the proof of Theorem \ref{theorem:measurable} shows that a cardinal $\kappa$ is measurable if and only if $\kappa$ is an \emph{outward compactness cardinal for $\LL^2_{\kappa,\omega}$}, where the latter is defined by replacing $\LL^2$ with $\LL^2_{\kappa,\omega}$ in Definition \ref{definition:outward compactnessSimple}. 
 Similar remarks will apply to all characterizations of large cardinal notions presented in this paper. 
 In our opinion, characterizations in which the cardinal $\kappa$ which is being characterized is not a parameter in the definition of the  logic that is being used for the characterization are more interesting. 
 Therefore, we will omit making these remarks in the remainder of the paper. 
 
 \item In the statement of Theorem \ref{theorem:measurable}, one could also replace $\LL^2$ by the logic $\LL(Q^{WF})$, as is obvious from the above proof. 
\end{enumerate}
\end{remark}

 In Sections \ref{section:PsiLarge} and \ref{section:CharLL2}, we will refine the notion of outward compactness in order to isolate analogous characterizations for other  large cardinal notions below extendibility. 
 The main result of these sections provides a general correspondence between objects in this region of the large cardinal hierarchy and outward compactness principles for second-order logic (see Theorem \ref{theorem:LL2Duality}) that we will afterwards apply to several well-studied large cardinal notions (see Corollary \ref{corollary:IndividualChar}). 
 This correspondence generalizes both Magidor's characterization of extendibility in  Theorem \ref{theorem:Magidor} and our characterization of measurability in Theorem \ref{theorem:measurable}. 
 Towards this goal, we will introduce  the more general notion of \emph{$\Psi$-outward compactness} for a first-order formula $\Psi$ in Definition \ref{definition:outward compactness} below. 
 This notion strengthens the concept of outward compactness by restricting the class of relevant outer $\ZFC^*$-models in Definition \ref{definition:outward compactnessSimple}.  
 We will introduce formulas $\Psi_{str}$, $\Psi_{sc}$, $\Psi_{stc}$ and $\Psi_{ext}$ that will allow us to obtain the following characterizations:

 \begin{theorem}
 \begin{enumerate}
  
  \item A cardinal $\kappa$ is a  $\Psi_{str}$-outward compactness cardinal for $\LL^2$ if and only if there  is a strong cardinal less than or equal to $\kappa$.  
    
   \item A cardinal $\kappa$ is a  $\Psi_{sc}$-outward compactness cardinal for $\LL^2$ if and only if there  is a supercompact cardinal less than or equal to $\kappa$. 
   
  \item A cardinal $\kappa$ is a  $\Psi_{stc}$-outward compactness cardinal for $\LL^2$ if and only if $\kappa$ is $\omega_1$-strongly compact. 
  
  \item A  cardinal $\kappa$ is a  $\Psi_{ext}$-outward compactness cardinal for $\LL^2$ if and only if there  is an extendible cardinal less than or equal to $\kappa$.  
 \end{enumerate}
\end{theorem}

 We then discuss the naturalness of the obtained large cardinal characterizations in Section \ref{section:Natural?}, and consider the question which other types of large cardinal properties below extendibility can be characterized through outward compactness principles. 
 Finally, in Sections \ref{section:FragmentsVP} and \ref{section:AbstractLogics}, we extend our concepts to arbitrary abstract logics (see Definition \ref{definition:AbstractLogic}) and large cardinal properties beyond extendibility. 
The main results of these sections will  provide a general correspondence between outward compactness for abstract logics and fragments of \emph{Vop\v{e}nka's Principle} (see Theorems \ref{theorem:ForwardGlobalPrinciples} and \ref{theorem:SchemesBackward}) that directly leads to outward compactness characterizations for  Vop\v{e}nka's Principle and the principle \anf{\emph{$\Ord$ is Woodin}} (see Corollaries \ref{corollary:OutwardVP} and \ref{corollary:OutwardOrdWoodin}). The given correspondence directly generalizes a classical result of Makowsky in \cite{MR780522} that characterizes the validity of Vop\v{e}nka's Principle through the existence of strong compactness cardinals for all abstract logics (see Theorem \ref{theorem:Makowsky}). 
These results will be based on the (even more general) notion of \emph{$F$-$\Psi$-outward compactness} introduced in Definition \ref{definition:FPsiOutward} below, that strengthens the concept of $\Psi$-outward compactness by requiring the  relevant outer models to satisfy stronger fragments $F$ of $\ZFC$.

 \begin{theorem}
  \begin{enumerate}
   \item The following schemes are equivalent over $\ZFC$: 
     \begin{enumerate}
       \item Vop\v{e}nka's Principle. 
    
       \item For every natural number $n$ and every abstract logic $\LL$, there exists a $\ZFC_n$-$\Psi_{ext}$-outward compactness cardinal for $\LL$. 
            
       \item For every natural number $n$ and every abstract logic $\LL$ with simple formulas, there exists a $\ZFC_n$-$\Psi_{sc}$-outward compactness cardinal for $\LL$. 
     \end{enumerate}
 
  \item The following schemes are equivalent over $\ZFC$: 
    \begin{enumerate}
       \item $\Ord$ is Woodin. 
    
       \item For every natural number $n$ and every abstract logic $\LL$ with simple formulas, there exists a $\ZFC_n$-$\Psi_{str}$-outward compactness cardinal for $\LL$. 
     \end{enumerate}
  \end{enumerate}
\end{theorem}


\section[Psi-large cardinals]{$\Psi$-large cardinals}\label{section:PsiLarge}

 We now introduce a broad framework for large cardinal properties between measurability and extendibility. Our results will later show that all notions that fall into this framework can be characterized through compactness properties of $\LL^2$. 
  In order to achieve this goal, we first define a property to express a certain amount of closeness between two transitive models $M\subseteq N$ of set theory. 
 In the following, we say that a class is \emph{closed under basic set operations} if it is closed under taking pairs, products and intersections. Note that for every class $M$ closed under basic set operations and every limit ordinal $\lambda$, the set $M\cap V_\lambda$ is closed under basic set operations.

\begin{definition}\label{definition:Closenessss}
 \begin{enumerate}
  \item A tuple $\langle N,M,\mu,\nu,\rho\rangle$ of sets is \emph{suitable} if $M\subseteq N$ are transitive sets closed under basic set operations with $M\cap\Ord=N\cap\Ord\in\Lim$ and   $\mu<\nu<\rho\in M$ are limit ordinals  with the property that  $\mu$ is a cardinal in $N$.  
  
  \item\label{item:Closeness} 
    A first-order formula $\Psi(v_0,\ldots,v_4)$ in the language of set theory is a \emph{measure of closeness} if $\ZFC$ proves the following statements:     \begin{enumerate}[label=\text{(\Alph*)}]
 
 \item\label{Prop:A} If $\Psi(N,M,\mu,\nu,\rho)$ holds, then the tuple $\langle N,M,\mu,\nu,\rho\rangle$ is suitable. 
  
   \item\label{Prop:B} If $\mu$ is a cardinal and $\nu<\rho<\theta$ are limit ordinals strictly greater than $\mu$, then  $\Psi(V_\theta,V_\theta,\mu,\nu,\rho)$ holds. 
   
     
  \item\label{Prop:F} Given a suitable tuple $\langle N,M,\mu,\nu,\rho\rangle$ and a limit ordinal $\theta$ with $\rho<\theta\in M$,  we have $$\Psi(N,M,\mu,\nu,\rho) ~ \longleftrightarrow ~ \Psi(N\cap V_\theta,M\cap V_\theta,\mu,\nu,\rho).$$
  \end{enumerate} 
 \end{enumerate}
\end{definition}

As a first example of the concept introduced above, let $\Psi_{ms}(v_0,\ldots,v_4)$ denote the canonical $\Delta_0$-formula in the language of set theory such that $\Psi_{ms}(N,M,\mu,\nu,\rho)$ holds if and only if  the tuple $\langle N,M,\mu,\nu,\rho\rangle$ is suitable.  
  It is easy to check that the formula $\Psi_{ms}$ also satisfies the above properties \ref{Prop:B} and \ref{Prop:F}:

  \begin{proposition}
   The formula $\Psi_{ms}$ is a measure of closeness. \qed 
  \end{proposition}

  We now continue by associating large cardinal properties to measures of closeness.

\begin{definition}
 Given a first-order formula  $\Psi(v_0,\ldots,v_4)$  in the language of set theory, 
  a cardinal $\kappa$ is \emph{$\Psi$-large}
   if for all limit ordinals $\kappa<\eta<\theta$, 
   there are unboundedly many  cardinals $\lambda\geq\eta$ with the property that there is  a transitive set $M$ and an elementary embedding $\map{j}{V_{\theta+1}}{M}$
    such that  $j(\kappa)>\lambda$ and $\Psi(V_{j(\theta)},M\cap V_{j(\theta)},\lambda,j(\kappa),j(\eta))$ holds. 
\end{definition}


 In the following, we want to show that several well-known large cardinal properties can be characterized through the notion of $\Psi$-largeness:  we can associate a measure of closeness $\Psi$ to the given large cardinal property and  prove that a cardinal $\kappa$ is $\Psi$-large if and only if there is a cardinal with the given large cardinal property that is less than or equal to~$\kappa$. Before we start to derive these characterizations, we observe that no large cardinal property that is implication-wise stronger than extendibility can be characterized in this way.

 \begin{proposition}\label{proposition:ExtendibleIsLargeForAll}
   If $\Psi(v_0,\ldots,v_4)$ is a measure of closeness, then every cardinal greater than or equal to an extendible cardinal is $\Psi$-large.  
 \end{proposition}
 
 \begin{proof}
   Assume that $\mu$ is an extendible cardinal, $\kappa\geq\mu$ is a cardinal,  $\kappa<\eta<\theta$ are limit ordinals and $\lambda>\theta$ is a cardinal.  
  Then there exists an ordinal $\rho$ and an elementary embedding $\map{i}{V_\lambda}{V_\rho}$ with $\crit(i)=\mu$ and $i(\mu)>\lambda$. 
  If we now define $$\map{j=i\restriction V_{\theta+1}}{V_{\theta+1}}{V_{i(\theta)+1}},$$ then $j(\kappa)>\lambda$ and  \ref{Prop:B} in Definition \ref{definition:Closenessss}.\ref{item:Closeness} ensures that $\Psi(V_{j(\theta)},V_{j(\theta)},\lambda,j(\kappa),j(\eta))$ holds. These computations  allow us to conclude that $\kappa$ is $\Psi$-large. 
 \end{proof}

 Next, we use standard arguments to show that $\Psi$-largeness can also be characterized through the existence  of elementary embeddings from $V$ into  inner models.

\begin{lemma}\label{lemma:EmbeddingLarge}
 Given a measure of closeness $\Psi(v_0,\ldots,v_4)$, a  cardinal $\kappa$ is $\Psi$-large if and only if for every limit ordinal $\eta>\kappa$, there are unboundedly many cardinals $\lambda\geq\eta$ with the property that there is an inner model $M$ and an elementary embedding $\map{j}{V}{M}$ such that $j(\kappa)>\lambda$ and  $\Psi(V_\vartheta,V_\vartheta^M,\lambda,j(\kappa),j(\eta))$ holds for all  limit ordinals $\vartheta>j(\eta)$. 
\end{lemma}

\begin{proof}
 First, assume that $\kappa$ is $\Psi$-large and fix  limit ordinals $\zeta>\eta>\kappa$. Pick a  sufficiently large fixed point $\nu>\zeta$ of the $\beth$-function. By our assumption, we find a cardinal $\lambda\geq\zeta$, a transitive set $M_0$, and an elementary embedding $\map{i}{V_{\nu+1}}{M_0}$, such that  $i(\kappa)>\lambda$ and $\Psi(V_{i(\nu)},M_0\cap V_{i(\nu)},\lambda,i(\kappa),i(\eta))$ holds. 
 Since $\nu$ is a strong limit cardinal, we can use standard arguments (see {\cite[Lemma 26.1]{kanamori}}) to find an inner model  $M$ with $V_{i(\nu)}^M=M_0\cap V_{i(\nu)}$, and an elementary embedding $\map{j}{V}{M}$ satisfying $$j\restriction(V_\nu\cup\{\nu\}) ~ = ~ i\restriction(V_\nu\cup\{\nu\}).$$  
That is, we define $E$ to be the $(\kappa,i(\nu))$-extender derived from $i$, defining $E_a$ for every $a\in[i(\nu)]^{<\omega}$ by \[X\in E_a\ \iff\ X\in\mathcal P([\nu]^{|a|})\,\land\,a\in i(X),\] and let $\map{j}{V}{M}$ be the ultrapower embedding of $V$ induced by $E$. Let $\map{j^\prime}{V_{\nu+1}}{M}^\prime$ be the ultrapower embedding of $V_{\nu+1}$ induced by $E$.
Since $V$ and $V_{\nu+1}$ have the same subsets of $\nu$, it follows that $M$ and $M^\prime$ agree up to $j(\nu)=j^\prime(\nu)$, and that $j$ and $j^\prime$ agree up to $\nu$. 
Let $\map{k}{M^\prime}{M_0}$ be the factor embedding satisfying that $i=k\circ j^\prime$.
Then, by \cite[Lemma 26.1]{kanamori}, we first obtain that $\crit(k)>j(\nu)$, yielding that $i(\nu)=j^\prime(\nu)=j(\nu)$, and then consequently that $V_{j(\nu)}^M=V_{j(\nu)}\cap M_0$. If $\vartheta>j(\eta)$ is a limit ordinal, we apply \ref{Prop:F} in Definition \ref{definition:Closenessss}.\ref{item:Closeness} to conclude that $\Psi(V_\vartheta,V_\vartheta^M,\lambda,j(\kappa),j(\eta))$ holds. Summing up, this shows that we have found our desired embedding $\map{j}{V}{M}$.

 In the other direction, assume that for some limit ordinal $\eta>\kappa$, there are unboundedly many cardinals $\lambda\geq\eta$ such that there exists an inner model $M$ and an elementary embedding $\map{j}{V}{M}$ for which $j(\kappa)>\lambda$ and  $\Psi(V_\vartheta,V_\vartheta^M,\lambda,j(\kappa),j(\eta))$ holds for all  limit ordinals $\vartheta>j(\eta)$. 
  Fix  a  limit ordinal $\theta>\eta$ and an ordinal $\zeta>\eta$. Pick a cardinal $\lambda\geq\zeta$, an inner model $M$ and   an elementary embedding $\map{j}{V}{M}$ for which $j(\kappa)>\lambda$ and  $\Psi(V_\vartheta,V_\vartheta^M,\lambda,j(\kappa),j(\eta))$ holds for all  limit ordinals $\vartheta>j(\eta)$. Then, $M_0=V^M_{j(\theta)+1}$ is a transitive set, and $\map{i=j\restriction V_{\theta+1}}{V_{\theta+1}}{M_0}$ is an elementary embedding for which $i(\kappa)>\lambda$ and $\Psi(V_{i(\theta)},M_0\cap V_{i(\theta)},\lambda,i(\kappa),i(\eta))$ holds. These computations yield the converse implication. 
\end{proof}

  Using the above equivalence, we can now easily show that the formula $\Psi_{ms}$ canonically corresponds to the large cardinal notion of measurability.

\begin{corollary}
 A cardinal $\kappa$ is $\Psi_{ms}$-large if and only if there is a measurable cardinal less than or equal to $\kappa$. 
\end{corollary}

\begin{proof}
 First, assume that $\kappa$ is $\Psi_{ms}$-large. By Lemma \ref{lemma:EmbeddingLarge}, there is a transitive class $M$ and an elementary embedding $\map{j}{V}{M}$ with $j(\kappa)>\kappa$. In particular, we know that the critical point $\crit(j)$ of $j$ is a measurable cardinal less than or equal to $\kappa$. 
 In the other direction, assume that there is a measurable cardinal less than or equal to $\kappa$ and  $\lambda>\kappa$ is a cardinal. Standard arguments about iterated ultrapowers (see {\cite[Corollary 19.7(b)]{kanamori}}) now allow us to find an inner model $M$ and an elementary embedding $\map{j}{V}{M}$ with $j(\kappa)>\lambda$. Then it is easy to see that $\Psi_{ms}(V_\vartheta,V_\vartheta^M,\lambda,j(\kappa),j(\eta))$ holds for every limit ordinal $\kappa<\eta\leq\lambda$ and every limit ordinal $\vartheta>j(\eta)$. By Lemma \ref{lemma:EmbeddingLarge}, this shows that the cardinal $\kappa$ is $\Psi_{ms}$-large.  
 \end{proof}

In the remainder of this section, we will show that several other important large cardinal notions in the region between measurability and extendibility canonically correspond to measures of closeness.

\subsection{Strong cardinals}\label{subsection:Strong}
 Let $\Psi_{str}(v_0,\ldots,v_4)$ denote the canonical first-order formula in the language of set theory with the property that $\Psi_{str}(N,M,\mu,\nu,\rho)$ holds if and only if the tuple $\langle N,M,\mu,\nu,\rho\rangle$ is suitable and $N\cap V_\mu\subseteq M$ holds. 
 The following statement is then immediate:

 \begin{proposition}
  The formula $\Psi_{str}$ is a measure of closeness. \qed 
 \end{proposition}

  In the following, we will show that the formula $\Psi_{str}$ corresponds to the large cardinal concept of strongness. These and later arguments will make crucial use of the following well-known consequence of the \emph{Kunen Inconsistency} (see, for example, {\cite[Corollary 23.14]{kanamori}}):

  \begin{lemma}\label{lemma:IteratedCrit}
   Let $M$ be a transitive class and let $\map{j}{V}{M}$ be a non-trivial elementary embedding. If $\lambda$ is a limit ordinal of uncountable cofinality with $V_\lambda\subseteq M$, then there exists a natural number $n$ with $j^n(\crit(j))\geq\lambda$. 
  \end{lemma}
  
  \begin{proof}
   Assume, towards a contradiction, that the above conclusion fails and define $$\zeta ~ = ~ \sup_{n<\omega}j^n(\crit(j)).$$ Then, $\zeta$ is a limit ordinal of countable cofinality with $j(\zeta)=\zeta$, and our assumptions ensure that $\zeta+2<\lambda$.  But this implies that the map $\map{j\restriction V_{\zeta+2}}{V_{\zeta+2}}{V_{\zeta+2}}$ is a non-trivial elementary embedding, contradicting the \emph{Kunen Inconsistency}. 
  \end{proof}

  \begin{lemma}\label{lemma:CharStrong}
   A cardinal $\kappa$ is $\Psi_{str}$-large if and only if there is a strong cardinal less than or equal to $\kappa$.
  \end{lemma}
  
  \begin{proof}
   First, assume that some cardinal $\mu\leq\kappa$ is strong, and $\lambda>\kappa$ is a cardinal. 
   In this situation, there exists an inner model $M$ and an elementary embedding $\map{j}{V}{M}$ with $\crit(j)=\mu$, $j(\mu)>\lambda$ and $V_\lambda\subseteq M$. Then, $j(\kappa)>\lambda$, and it is easy to see that $\Psi_{str}(V_\vartheta,V_\vartheta^M,\lambda,j(\kappa),j(\eta))$ holds for  all limit ordinals $\kappa<\eta\leq\lambda$ and all limit ordinals $\vartheta>j(\eta)$.  By Lemma   \ref{lemma:EmbeddingLarge}, this shows that $\kappa$ is $\Psi_{str}$-large.

 In the other direction, assume that $\kappa$ is $\Psi_{str}$-large, and no cardinal less than or equal to $\kappa$ is strong. Then, there exists a regular cardinal $\eta>\kappa$ with the property that no cardinal less than or equal to $\kappa$ is $\eta$-strong. 
  Using Lemma \ref{lemma:EmbeddingLarge}, we find a cardinal $\lambda>\eta$, an inner model $M$ with $V_\lambda\subseteq M$ and an elementary embedding $\map{j}{V}{M}$ with $j(\kappa)>\lambda$. Set $\mu=\crit(j)<\kappa$ and $\mu_i=j^i(\mu)$ for all $i<\omega$. Lemma \ref{lemma:IteratedCrit} yields a least natural number $n$ with $\mu_n>\eta$.  Since $\mu<\eta$, we know that $n>0$, and the fact that $\mu$ is not $\eta$-strong then implies that $n>1$.

  \begin{claim*}
   In $M$, the cardinal $\mu_1$ is superstrong with target $\mu_n$.\footnote{Remember that a cardinal $\kappa$ is \emph{superstrong with target $\lambda$} if there is a transitive class $M$ with $V_\lambda\subseteq M$ and an elementary embedding $\map{j}{V}{M}$ with $\crit(j)=\kappa$ and $j(\kappa)=\lambda$.} 
  \end{claim*}
  
  \begin{proof}[Proof of the Claim]
   By induction, we show that, in $M$,   for all $0<i<n$, the cardinal $\mu_1$ is superstrong with target $\mu_{i+1}$. 
   Since $n>1$, the embedding $j$ witnesses that $\mu$ is superstrong with target $\mu_1$ in $V$ and hence elementarity of $j$ ensures that, in $M$, the cardinal $\mu_1$ is superstrong with target $\mu_2$. Next, assume that $i<n-1$ has the property that  the cardinal $\mu_1$ is superstrong with target $\mu_{i+1}$  in $M$. 
   Then there exists a transitive class $N\subseteq M$ with $V^M_{\mu_{i+1}}\subseteq N$ and an elementary embedding $\map{i}{M}{N}$ with $\crit(i)=\mu_1$ and $i(\mu_1)=\mu_{i+1}$. 
   Since $i+1<n$, we have  $\mu_{i+1}\leq\eta$ and hence $V_{\mu_{i+1}}\subseteq N$. 
   Moreover, the map $\map{i\circ j}{V}{N}$ is an elementary embedding with $\crit(i\circ j)=\mu$ and $(i\circ j)(\mu)=\mu_{i+1}$. In particular, this map witnesses that $\mu$ is superstrong with target $\mu_{i+1}$ in $V$. This allows us to conclude that, in $M$, the cardinal $\mu_1$ is superstrong with target  $\mu_{i+2}$. 
   This argument completes the proof of the claim. 
  \end{proof}

  By the above claim, there exists a transitive class $N\subseteq M$ with $V^M_{\mu_n}\subseteq N$ and an elementary embedding $\map{i}{M}{N}$ with $\crit(i)=\mu_1$ and $i(\mu_1)=\mu_n$. But then, $V_\lambda\subseteq N$ and $\map{i\circ j}{V}{N}$ is an elementary embedding with $\crit(i\circ j)=\mu$ and $(i\circ j)(\mu)>\eta$, yielding that $\mu$ is $\eta$-strong, a contradiction.   
  \end{proof}


\subsection{Supercompact cardinals}

 Let $\Psi_{sc}(v_0,\ldots,v_4)$ denote the canonical first-order formula in the language of set theory with the property that $\Psi_{sc}(N,M,\mu,\nu,\rho)$ holds if and only if the tuple $\langle N,M,\mu,\nu,\rho\rangle$ is suitable and ${}^\mu\rho\cap N\subseteq M$ holds.

  \begin{proposition}
  The formula $\Psi_{sc}$ is a measure of closeness. \qed 
 \end{proposition}

 We now  show that this formula corresponds to the large cardinal concept of supercompactness.

 \begin{lemma}\label{lemma:PsiLargeSupercompact}
  A cardinal $\kappa$ is $\Psi_{sc}$-large if and only if there is a supercompact cardinal less than or equal to $\kappa$. 
 \end{lemma}
 
 \begin{proof}
  First, assume that $\mu\leq\kappa$ is a supercompact cardinal, and $\lambda>\kappa$ is a cardinal. Then, there exists a transitive class $M$ with ${}^\lambda M\subseteq M$ and an elementary embedding $\map{j}{V}{M}$ with $\crit(j)=\mu$ and $j(\mu)>\lambda$. 
  If $\kappa<\eta\leq\lambda$ is a limit ordinal, and  $\vartheta>j(\eta)$ is a limit ordinal, then ${}^{{<}\lambda}j(\eta)\subseteq V_\vartheta^M$ and this shows that $\Psi_{sc}(V_\vartheta,V_\vartheta^M,\lambda,j(\kappa),j(\eta))$  holds. By Lemma \ref{lemma:EmbeddingLarge}, this shows that $\kappa$ is $\Psi_{sc}$-large.

  Now, assume that $\kappa$ is $\Psi_{sc}$-large, and no cardinal less than or equal to $\kappa$ is supercompact. Pick a fixed point $\eta>\kappa$ of the $\beth$-function of uncountable cofinality with the property that no cardinal less than or equal to $\kappa$ is $\eta$-supercompact, and apply Lemma \ref{lemma:EmbeddingLarge} to find a cardinal $\lambda>\eta$, a transitive class $M$ with ${}^\lambda j(\eta)\subseteq M$ and an elementary embedding $\map{j}{V}{M}$ with $j(\kappa)>\lambda$. 
  %
  %
  Define $\mu=\crit(j)<\kappa$, and  $\mu_i=j^i(\mu)$ for all $i<\omega$. Then, our choice of $\eta$ ensures that $V_\eta\subseteq M$, and we can use Lemma \ref{lemma:IteratedCrit} to find a minimal natural number $n>1$ with $\mu_n>\eta$.

  \begin{claim*}
   The cardinal $\mu$ is (n-1)-huge with targets $\mu_1,\ldots,\mu_{n-1}$.\footnote{Remember that, given a natural number $n>0$, a cardinal $\kappa$ is \emph{$n$-huge with targets $\lambda_1,\ldots,\lambda_n$} if there is a transitive class $M$ with ${}^{\lambda_n}M\subseteq M$ and an elementary embedding $\map{j}{V}{M}$ with $\crit(j)=\kappa$ and $j(\lambda_m)=\lambda_{m+1}$ for all $m<n$.} 
  \end{claim*}
  
  \begin{proof}
   Since we have $\mu_{n-1}\leq\eta<\lambda$, the fact that ${}^\lambda j(\eta)\subseteq M$ implies that $j[\mu_{n-1}]\in M$, and hence, we know that $$\Set{A\subseteq\Pow(\mu_{n-1})}{j[\mu_{n-1}]\in j(A)}$$ is a ${<}\mu$-complete normal ultrafilter over $\Pow(\mu_{n-1})$ that contains the set $$\Set{a\subseteq \mu_{n-1}}{\text{$\ot(a\cap\mu_i)=\mu_{i-1}$ for all $0<i<n$}}.$$ By {\cite[Theorem 24.8]{kanamori}}, this implies the statement of the claim. 
  \end{proof}

  By the above claim, it easily follows  that $\mu$ is huge with target $\mu_{n-1}$, and hence, elementarity implies that, in $M$, the cardinal $\mu_1$ is huge with target $\mu_n$. 
  Therefore, we find a transitive class $N\subseteq M$ definable in $M$ with $M\cap{}^{\mu_n}N\subseteq N$, and an elementary embedding $\map{i}{M}{N}$ definable in $M$ with $\crit(i)=\mu_1$ and $i(\mu_1)=\mu_n$. 
 Since $j[\eta]\in{}^\eta j(\eta)\subseteq M$ and $\eta<\mu_n$, we then know that $$(i\circ j)[\eta] ~ = ~ i[j[\eta]] ~ \in ~ M\cap {}^\eta N ~ \subseteq ~ N.$$ 
  This shows that $\map{i\circ j}{V}{N}$ is an elementary embedding with $\crit(i\circ j)=\mu$, $(i\circ j)(\mu)>\eta$ and $(i\circ j)[\eta]\in N$. But this contradicts the fact that $\mu$ is not $\eta$-supercompact.    
 \end{proof}


\subsection{Extendible cardinals}

 Let $\Psi_{ext}(v_0,\ldots,v_4)$ denote the canonical first-order formula in the language of set theory with the property that $\Psi_{ext}(N,M,\mu,\nu,\rho)$ holds if and only if the tuple $\langle N,M,\mu,\nu,\rho\rangle$ is suitable 
  and $N\cap V_\rho\subseteq M$ holds.

  \begin{proposition}
  The formula $\Psi_{ext}$ is a measure of closeness. \qed 
 \end{proposition}

 The following result shows that the formula $\Psi_{ext}$ corresponds to extendibility. In order to motivate the concepts that will be introduced in the next  chapter, we use Theorem \ref{theorem:Magidor} to prove the backward implication of the stated equivalence.\footnote{It would also be possible to establish this equivalence with the help of {\cite[Theorem 2.28]{MR3151400}}, which shows that a cardinal $\kappa$ is extendible if and only if it is \emph{jointly $\lambda$-supercompact and $\lambda$-superstrong} for all $\lambda>\kappa$ (see {\cite[Definition 2.24]{MR3151400}}). However, the proof of the equivalence that we present, using the connections between second-order logic and extendibility, provides an elegant alternative to an adaptation of the proof of Lemma \ref{lemma:PsiLargeSupercompact} to these large cardinal properties.}

 \begin{lemma}\label{lemma:ExtLarge}
  A cardinal $\kappa$ is $\Psi_{ext}$-large if and only if there is an extendible cardinal less than or equal to $\kappa$. 
 \end{lemma}
 
 \begin{proof}
  First, note that by Proposition \ref{proposition:ExtendibleIsLargeForAll}, every cardinal greater than or equal to an extendible cardinal is $\Psi_{ext}$-large. 
  In the other direction, assume that $\kappa$ is $\Psi_{ext}$-large and let $T$ be a ${<}\kappa$-satisfiable $\LL^2$-theory. Pick a cardinal $\eta>\kappa$ with $T\in V_\eta\prec_{\Sigma_2}V$. By our assumption, we can apply Lemma \ref{lemma:EmbeddingLarge} to find a cardinal $\lambda>\eta$,  a transitive class $M$ and an elementary embedding $\map{j}{V}{M}$ such that $j(\kappa)>\lambda$  
   and $V_{j(\eta)}\subseteq M$. 
  We then know that $V_{j(\eta)}=V^M_{j(\eta)}\prec_{\Sigma_2}M$, and $j(T)\in V_{j(\eta)}$ is a ${<}j(\kappa)$-satisfiable $\LL^2$-theory in $M$. 
  Moreover, our assumptions on $M$ ensure that $j[T]$ is an element of $M$ and has cardinality less than $j(\kappa)$ in $M$. But this implies that $j[T]$ is satisfiable in $M$, and the above observations show that $V_{j(\eta)}$ contains a structure that is a model of $j[T]$ in $M$. We conclude that $j[T]$ is also satisfiable in $V$ and, since we can identify the theories $T$ and $j[T]$ through a renaming of symbols, this shows that $T$ is satisfiable in $V$. These computations show that $\kappa$ is a strong compactness cardinal for $\LL^2$, and hence, Theorem~\ref{theorem:Magidor} ensures that there exists an extendible cardinal less than or equal to $\kappa$.  
 \end{proof}


\subsection{$\omega_1$-strongly compact cardinals}\label{subsection:Omega1StronglyCompact}
 Remember that a cardinal $\kappa>\omega_1$ is \emph{$\omega_1$-strongly compact} if for every set $I$, every ${<}\kappa$-complete filter on $I$ can be extended to a countably complete ultrafilter on $I$ (see \cite{MR3152715}, \cite{MR3226024} and \cite{MR926461}). Note that every cardinal above an $\omega_1$-strongly compact cardinal is also $\omega_1$-strongly compact. 
 In the following, we show that the na\"ive attempt to characterize strongly compact cardinals through the concept of $\Psi$ -largeness (motivated by {\cite[Theorem 22.17]{kanamori}}) actually leads to a characterization of $\omega_1$-strongly compact cardinals. 
 This characterization differs strongly from the ones presented above, because, {e.g.,} the first $\omega_1$-strongly compact cardinal can be singular (see {\cite[Theorem 6.1]{MR3152715}}). 
  Let $\Psi_{stc}(v_0,\ldots,v_4)$ denote the canonical first-order formula in the language of set theory with the property that $\Psi_{stc}(N,M,\mu,\nu,\rho)$ holds if and only if the tuple $\langle N,M,\mu,\nu,\rho\rangle$ is suitable   and   for every $d\in N\cap\Pow(\rho)$ with the property that $N$ contains no injection of $\mu$ into $d$, there exists $c\in M\cap\Pow(\rho)$ such that $d\subseteq c$ and $M$ does not contain an injection of $\nu$ into $c$.\footnote{Note that these assumptions are phrased in this way because Definition \ref{definition:Closenessss}.\ref{item:Closeness} only makes minimal assumptions on the closure properties of the sets $M$ and $N$.}

  \begin{proposition}
   The formula $\Psi_{stc}$ is a measure of closeness. 
  \end{proposition}
  
  \begin{proof}
    %
   %
The formula $\Psi_{stc}$ trivially satisfies \ref{Prop:A} and \ref{Prop:B} in Definition \ref{definition:Closenessss}.\ref{item:Closeness}. 
Assume that the tuple $\langle N,M,\mu,\nu,\rho\rangle$ is suitable and $\theta>\rho$ is a limit ordinal in $M$.  
    Since $M$ and $M\cap V_\theta$ (respectively, $N$ and $N\cap V_\theta$) contain the same subsets of $\rho$, the same injections from $\mu$ into $\rho$ and the same injections from $\nu$ into $\rho$, it follows directly that $\Psi_{stc}(N,M,\mu,\nu,\rho)$ holds if and only if $\Psi_{stc}(N\cap V_\theta,M\cap V_\theta,\mu,\nu,\rho)$ holds. This shows that $\Psi_{stc}$ also  satisfies \ref{Prop:F} in Definition \ref{definition:Closenessss}.\ref{item:Closeness}.    
  \end{proof}

 \begin{lemma}
  A cardinal is $\omega_1$-strongly compact if and only if it is $\Psi_{stc}$-large. 
 \end{lemma}
 
 \begin{proof}
  First, assume that $\kappa$ is an $\omega_1$-strongly compact cardinal, and let $\lambda>\kappa$ be a cardinal. 
  The proof of  {\cite[Theorem 4.7]{MR3152715}} then shows that there is a $\sigma$-complete fine ultrafilter $\UU$ on $\Pow_{\kappa}(\lambda)$. Let $\map{j}{V}{M}$ denote the ultrapower embedding induced by $\UU$, and fix $d\subseteq j(\lambda)$ of cardinality at most $\lambda$. Then there exists a sequence $$\seq{\map{f_\gamma}{\Pow_\kappa(\lambda)}{\lambda}}{\gamma<\lambda}$$ of functions with $d=\Set{[f_\gamma]_\UU}{\gamma<\lambda}$. Define $$\Map{F}{\Pow_\kappa(\lambda)}{\Pow_\kappa(\lambda)}{a}{\Set{f_\gamma(a)}{\gamma\in a}}.$$ Given $\gamma<\lambda$, we then have $$\Set{a\in \Pow_\kappa(\lambda)}{f_\gamma(a)\in F(a)} ~ \supseteq ~ \Set{a\in \Pow_\kappa(\lambda)}{\gamma\in a} ~ \in ~ \UU,$$ and this implies that $[f_\gamma]_\UU\in[F]_\UU$. Moreover, since $F(a)$ has cardinality less than $\kappa$ for each $a\in \Pow_\kappa(\lambda)$, we know that $[F]_\UU$ has cardinality less than $j(\kappa)$  in $M$. 
  In particular, if $d$ has cardinality  $\lambda$, then $\lambda=\vert d\vert\leq\vert d\vert^M\leq\vert [F]_\UU\vert^M<j(\kappa)$. 
  These computations show that for every limit ordinal $\kappa<\eta\leq\lambda$ and every $d\in\Pow_\lambda(j(\eta))$, there is $c\in\Pow_{j(\kappa)}(j(\eta))^M$ with $d\subseteq c$. 
 This directly implies that   $\Psi_{stc}(V_\vartheta,V_\vartheta^M,\lambda,j(\kappa),j(\eta))$ holds for every limit ordinal $\kappa<\eta\leq\lambda$ and every limit ordinal $\vartheta>j(\eta)$.  Using Lemma \ref{lemma:EmbeddingLarge}, this allows us to conclude that $\kappa$ is $\Psi_{stc}$-large.

  Now, assume that $\kappa$ is $\Psi_{stc}$-large, and let $\eta>\kappa$  be a limit ordinal. 
  Lemma \ref{lemma:EmbeddingLarge} allows us to find a cardinal $\lambda>\eta$,  a transitive class $M$ and an elementary embedding $\map{j}{V}{M}$ such that $j(\kappa)>\lambda$ and $\Psi_{stc}(V_\vartheta,V_\vartheta^M,\lambda,j(\kappa),j(\eta))$ holds for all  limit ordinals $\vartheta>j(\eta)$. 
  Since $j[\eta]$ is a subset of $j(\eta)$ of cardinality less than $\lambda$, we can now find $b\in\Pow(j(\eta))^M$ such that $j[\eta]\subseteq b$ and~$b$ has cardinality less than $j(\kappa)$ in $M$. By {\cite[Theorem 4.7]{MR3152715}}, these computations show that  $\kappa$ is $\omega_1$-strongly compact.  
 \end{proof}


\section{$\LL^2$-characterizations}\label{section:CharLL2}

We now connect the concepts introduced in the previous section to compactness properties of second-order logic. The next definition provides the first step to establish these connections. 
 In this definition, we consider simply definable measures of closeness with the additional property that the closeness of a transitive model $M$ to $V_{M\cap\Ord}$ is expressible over $M$ by an $\LL^2$-statement.  
Remember that, given a theory $F$ in the language of set theory and a natural number $n>0$, a set-theoretic formula $\varphi(v_0,\ldots,v_{m-1})$ is a \emph{$\Delta_n^F$-formula} if there is a $\Sigma_n$-formula $\psi_0(v_0,\ldots,v_{m-1})$ and a $\Pi_n$-formula $\psi_1(v_0,\ldots,v_{m-1})$ with the property that $$F ~ \vdash ~ \forall x_0,\ldots,x_{m-1} ~ [\varphi(x_0,\ldots,x_{m-1})\longleftrightarrow\psi_i(x_0,\ldots,x_{m-1})]$$ holds for all $i<2$.

\begin{definition}\label{definition:moc}
 A first-order formula $\Psi(v_0,\ldots,v_4)$ in the language of set theory is an \emph{$\LL^2$-measure of closeness} if the following statements hold: 
 \begin{enumerate}
  \item\label{item:L2measure1} $\Psi$ is  a measure of closeness. 
  
  \item\label{item:L2measure2} $\Psi$ is both a $\Delta_1^{\ZFC^*}$- and a $\Delta_1^{\ZFC^-}$-formula.\footnote{As usual, we let $\ZFC^-$ denote the axioms of $\ZFC$ without the Powerset axiom and with the Collection scheme instead of the Replacement scheme.}
  
  \item\label{item:L2measure3} There exists an $\LL^2$-formula $\Psi^*(v_0,v_1,v_2)$ in the language of set theory such that $\ZFC^*$ proves that for all  limit ordinals $\mu<\nu<\rho<\theta$ and every  transitive set $M$ with $\theta\subseteq M\subseteq V_\theta$, the statement $\Psi(V_\theta,M,\mu,\nu,\rho)$ holds if and only if $\langle M,\in\rangle\models_{\LL^2}\Psi^*(\mu,\nu,\rho)$ holds. 
 \end{enumerate}
\end{definition}

\begin{proposition}\label{proposition:FormulasAreL2measures}
 The formulas $\Psi_{ms}$, $\Psi_{str}$, $\Psi_{sc}$, $\Psi_{stc}$ and $\Psi_{ext}$ are all $\LL^2$-measures of closeness. 
\end{proposition}

\begin{proof}
 First, note that it is easy to find an $\LL^2$-formula $\varphi(v)$ with the property that $\ZFC^*$ proves that for every  transitive set~$M$ closed under basic set operations and every $a\in M$, we have $\langle M,\in\rangle\models_{\LL^2}\varphi(a)$ if and only if $a$ is an element of $M\cap\Ord$ that is a infinite cardinal in $V$.\footnote{This uses that, since $M$ is closed under forming ordered pairs, any functions between ordinals in $M$ is a subset of $M$. In the remainder of this paper, we will omit remarking on similar usage of closure under basic set operations.}
 Since $\Psi_{ms}$ is a $\Delta_0$-formula,  this shows that $\Psi_{ms}$ is an $\LL^2$-measure of closeness.

 Next, there is a canonical $\LL^2$-formula $\varphi(v_0,v_1)$ such that $\ZFC^*$ proves that for every  transitive set $M$ closed under basic set operations and all $\alpha,\beta\in M\cap\Ord$, we have  $\langle M,\in\rangle\models_{\LL^2}\varphi(\alpha,\beta)$ if and only if the set ${}^\alpha\beta$ is a subset of $M$. Since $\Psi_{sc}$ is a $\Delta_0$-formula, it follows that $\Psi_{sc}$ is an $\LL^2$-measure of closeness.

  Now, observe that there is an  $\LL^2$-formula $\varphi(v_0,v_1,v_2)$ such that $\ZFC^*$ proves that for every  transitive set $M$ closed under basic set operations and all $\alpha,\beta,\gamma\in M\cap\Ord$, we have $\langle M,\in\rangle\models_{\LL^2}\varphi(\alpha,\beta,\gamma)$ if and only if for every subset $d$ of $\gamma$ of cardinality less than~$\alpha$, there exists $c\in M$ such that $d\subseteq c\subseteq\gamma$ and $M$ contains no injection from $\beta$ into $c$. Since $\Psi_{stc}$ is a $\Delta_0$-formula, this again shows that $\Psi_{stc}$ is an $\LL^2$-measure of closeness.

  For the remaining formulas, recall that  an argument of Magidor in {\cite[Proof of Theorem~2]{MR0295904}} (see also {\cite[Proof of Fact 2.1]{MR4093885}}) shows that  there is an $\LL^2$-sentence $\psi$ in the language of set theory with the property that for every transitive set $M$, we have $\langle M,\in\rangle\models_{\LL^2}\psi$ if and only if $M=V_{M\cap\Ord}$. 
  In addition, Magidor's argument shows that this equivalence is provable in $\ZFC^*$ 
  %
  and, moreover, it directly provides an $\LL^2$-formula $\varphi(v_0,v_1)$ in the language of set theory with the property that $\ZFC^*$ proves that for every transitive set $M$ and all $a,b\in M$, we have $\langle M,\in\rangle\models_{\LL^2}\varphi(a,b)$   if and only if $a\in M\cap\Ord$ and $b=V_a$. In combination with the fact that  the rank function is definable by a $\Delta_1^{\ZFC^*\cap\ZFC^-}$-formula, we can conclude that the formulas $\Psi_{str}$ and $\Psi_{ext}$ are both $\LL^2$-measures of closeness. 
\end{proof}

The second step to connect $\Psi$-large cardinals to second-order logic is given by the family of compactness properties introduced in the next definition.

\begin{definition}\label{definition:outward compactness}
Let $\Psi(v_0,\ldots,v_4)$ be a first-order formula in the language of set theory and let $\kappa$ be an infinite cardinal. 
 \begin{enumerate}
  \item An $\LL^2$-theory $T$ is \emph{$\Psi$-outward satisfiable at $\kappa$} if there is a limit ordinal $\eta>\kappa$ such that 
  for all infinite cardinals $\lambda<\kappa$ and 
   all cardinals $\vartheta>\eta$ with $T\in H_\vartheta$, 
   the theory  $T$ is ${<}\lambda$-satisfiable in $N$, whenever $G$ is $\Coll(\omega,\vartheta)$-generic over $V$ and $N\in V[G]$ is an outer $\ZFC^*$-model     of $V_\vartheta^V$ with the property that $\Psi(N,V_\vartheta^V,\lambda,\kappa,\eta)$ holds in $V[G]$ .\footnote{Note that this can be written as a first-order statement in the parameters $T$ and $\kappa$.}

  
  \item A cardinal $\kappa$ is a \emph{$\Psi$-outward compactness cardinal for $\LL^2$} if all $\LL^2$-theories that are $\Psi$-outward satisfiable at $\kappa$ are satisfiable.
 \end{enumerate}
\end{definition}

As noted after Definition \ref{definition:outward compactnessSimple} above, if $\vartheta$ is a limit ordinal, $G$ is $\Coll(\omega,\vartheta)$-generic over $V$ and $N$ is an outer $\ZFC^*$-model of $V_\vartheta^V$ in $V[G]$, then both $N$ and $V_\vartheta^V$ are countable in $V[G]$.  
 Moreover, it is easy to see that, given  an $\LL^2$-theory $T$ and an infinite cardinal $\kappa$, the theory $T$ is ${<}\kappa$-outward satisfiable (as defined in Definition \ref{definition:outward compactnessSimple}) if and only if it is $\Psi_{ms}$-outward satisfiable at $\kappa$. 
 In the following, we make a further basic observation about the above notions that relate them to strong compactness cardinals for $\LL^2$.

 \begin{proposition}\label{proposition:OutwardLessCompact}
  Let $\Psi(v_0,\ldots,v_4)$ be a $\Delta_1^{\ZFC}$-formula that is a measure of closeness. If  $\kappa$  is a limit cardinal, then every $\LL^2$-theory that is $\Psi$-outward satisfiable at $\kappa$ is ${<}\kappa$-satisfiable.  
  In particular, every cardinal greater than or equal to an extendible cardinal is a $\Psi$-outward compactness cardinal for $\LL^2$. 
 \end{proposition}
 
 \begin{proof}
  Let $T$ be an $\LL^2$-theory that contains an unsatisfiable subtheory $T_0$ of cardinality less than $\kappa$, and let $\eta>\kappa$ be a limit ordinal. Pick a cardinal $\vartheta>\eta$ such that $T\in H_\vartheta=V_\vartheta$ and $V_\vartheta$ is sufficiently elementary in $V$. 
   In addition, pick a cardinal $\lambda<\kappa$ with $\vert T_0\vert<\lambda$. Then \ref{Prop:B} in Definition \ref{definition:Closenessss}.\ref{item:Closeness} ensures that $\Psi(V_\vartheta,V_\vartheta,\lambda,\kappa,\eta)$ holds. 
  Moreover, since $T$ is an element of $V_\vartheta$ and $V_\vartheta$ was chosen to be sufficiently elementary in $V$, we also know that $T$ is not ${<}\lambda$-satisfiable in $V_\vartheta$. 
  Now, let $G$ be $\Coll(\omega,\vartheta)$-generic over $V$.  Then  $V_\vartheta^V$ is an outer $\ZFC^*$-model of   $V_\vartheta^V$ in $V[G]$. Moreover, since $\Psi$ is a $\Delta_1^{\ZFC}$-formula, we can use $\Sigma_1$-upwards absoluteness to infer that $\Psi(V_\vartheta^V,V_\vartheta^V,\lambda,\kappa,\eta)$ holds in $V[G]$. Therefore, the theory $T$ is not $\Psi$-outward satisfiable at $\kappa$ in $V$.     
 \end{proof}

 We now combine the above two concepts to prove a general duality theorem that connects $\Psi$-largeness with $\Psi$-outward compactness. Together with our earlier results, this can be seen as a generalization of Theorem \ref{theorem:measurable} (see also Corollary \ref{corollary:IndividualChar} below).
  For the proof of this result, we  inductively define \emph{the second-order relativisation $\Psi\vert v$}  of an $\LL^2$-formula $\Psi$ in the language of set theory to some first-order variable $v$ that does not appear in $\Psi$ by setting: 
  \begin{itemize} 
   \item  ${(v_0\in v_1) \vert v} ~ \equiv ~ v_0\in v_1\in v$.
    \item  ${(u\in W)\vert v} ~ \equiv ~ u\in W\subseteq v$.  
    \item ${(\neg\varphi)\vert v} ~ \equiv ~ {\neg(\varphi\vert v)}$. 
    \item   ${(\varphi\wedge\psi)\vert v} ~ \equiv ~ {(\varphi\vert v)\wedge(\psi\vert v)}$. 
     \item  ${(\forall x ~ \varphi(x))\vert v} ~ \equiv ~ {\forall x\in v ~ (\varphi\vert v)(x)}$. 
     \item  ${(\forall X ~ \varphi(X))\vert v} ~ \equiv ~ {\forall X\subseteq v ~ (\varphi\vert v)(X)}$. 
   \end{itemize}
  The resulting formulas in this inductive definition are $\LL^2$-formulas in each case. 
  Moreover, if $\varphi(v_0,\ldots,v_{n-1})$ is an $\LL^2$-formula, then $\ZFC^*$ proves that $\langle M,\in\rangle\models_{\LL^2}(\varphi\vert v_n)(c_0,\ldots,c_{n-1},d)$ is equivalent to $\langle d,\in\rangle\models_{\LL^2}\varphi(c_0,\ldots,c_{n-1})$      whenever $M$ is a  transitive set, $d\in M$ is transitive and $c_0,\ldots,c_{n-1}\in d$.

\begin{theorem}\label{theorem:LL2Duality}
  Given  an  $\LL^2$-measure of closeness $\Psi$,  an infinite cardinal $\kappa$ is $\Psi$-large if and only if it is a $\Psi$-outward compactness cardinal for $\LL^2$.  
\end{theorem}

\begin{proof}
  First, assume that $\kappa$ is $\Psi$-large. Let  $T$ be an $\LL^2$-theory that is  $\Psi$-outward satisfiable at $\kappa$, and let $\eta>\kappa$ be a limit ordinal witnessing this. 
  Using Lemma \ref{lemma:EmbeddingLarge}, we can find a cardinal $\lambda>\eta$ with $T\in H_\lambda$, an inner model $M$ and an elementary embedding $\map{j}{V}{M}$ such that $j(\kappa)>\lambda$ and $\Psi(V_\vartheta,V_\vartheta^M,\lambda,j(\kappa),j(\eta))$ holds for all limit ordinals $\vartheta>j(\eta)$. 
     Pick a sufficiently large cardinal $\vartheta>j(\lambda)$ such that $j(\vartheta)=\vartheta$, $H_\vartheta=V_\vartheta$ and  the set $V_\vartheta$ is a model 
     of $\ZFC^*$. 
  Elementarity now implies that, in $M$, the cardinal $j(\eta)$ witnesses that the $\LL^2$-theory $j(T)$ is $\Psi$-outward satisfiable at $j(\kappa)$.   
  Let $G$ be $\Coll(\omega,\vartheta)$-generic over $V$. 
  Since $V_\vartheta^M$ is a model of $\ZFC^*$, 
  $V_\vartheta^M$ is countable in $M[G]$.  Moreover,    our setup ensures that, in $M[G]$, the $\LL^2$-theory $j(T)$ is ${<}\lambda$-satisfiable in every countable outer $\ZFC^*$-model $N$ of $V_\vartheta^M$ with the property that $\Psi(N,V_\vartheta^M,\lambda,j(\kappa),j(\eta))$ holds.   
    Since our assumptions on $\Psi$ imply that this statement can be formulated by a $\Pi_1$-formula with parameters in $H_{\aleph_1}^{M[G]}$ and is therefore provably equivalent to a $\mathbf{\Pi}^1_2$-statement whose parameters are real numbers canonically coding the original parameters (see {\cite[Lemma 25.25]{MR1940513}}), \emph{Shoenfield absoluteness} implies that  the given statement also holds in $V[G]$. 
  But $\Sigma_1$-upwards  absoluteness implies that $\Psi(V_\vartheta^V,V_\vartheta^M,\lambda,j(\kappa),j(\eta))$ also holds in $V[G]$ and, since $V_\vartheta^V$ is a countable outer $\ZFC^*$-model of $V_\vartheta^M$  in $V[G]$, we conclude that  $j(T)$ is ${<}\lambda$-satisfiable in $V_\vartheta^V$. 
   Now,  note that $j[T]\subseteq j(T)$ is an element of $V_\vartheta^V$, and this set has cardinality less than $\lambda$ in $V_\vartheta^V$.  Thus, the theory $j[T]$ is satisfiable in $V_\vartheta^V$ and, since $V_\vartheta^V$ correctly computes the power sets of its elements with respect to $V$, 
   the nature of the satisfaction relation of $\LL^2$ ensures that  $j[T]$ is a satisfiable $\LL^2$-theory in $V$. But, this also shows that $T$ is satisfiable in $V$, because the finitary character of $\LL^2$-formulae ensures that we can identify $T$ and $j[T]$ via the renaming of symbols induced by $j$.

  Now, assume that $\kappa$ is a $\Psi$-outward compactness cardinal for $\LL^2$. Let $\theta>\eta>\kappa$ be limit ordinals and let $\zeta>\eta$ be a cardinal.  In addition, let $\Psi^*(v_0,v_1,v_2)$ be the $\LL^2$-formula corresponding to $\Psi$ as in Definition \ref{definition:moc}.\ref{item:L2measure3} and 
  let $\Psi^\prime(v_0,\ldots,v_3)$ denote the second-order relativisation $\Psi^*\vert v_3$ of $\Psi^*(v_0,v_1,v_2)$ to $v_3$.
  Consider the language that extends the language of set theory by a constant symbol $b$, constant symbols $c_x$ for all elements $x$ of $V_{\theta+1}$ and   constant symbols $d_\gamma$ for all $\gamma\leq\zeta$.  Let $T$ denote the second-order theory consisting of the following:
  \begin{enumerate}
    \item The first-order elementary diagram of $V_{\theta+1}$, using the constant symbols $c_x$ for $x$ in $V_{\theta+1}$. 
    
    \item All first-order sentences of the form $\anf{d_\beta<d_\gamma<c_\kappa}$ for $\beta<\gamma\leq\zeta$. 
    
    \item The second-order sentence stating  that the $\in$-relation is well-founded. 
    
    \item The first-order sentence $\anf{d_\zeta<b<c_\kappa}$. 
    
    
    \item The second-order sentence $\Psi^\prime(b,c_\kappa,c_\eta,c_{V_\theta})$.  
  \end{enumerate}

  \begin{claim*}
   The ordinal $\eta$ witnesses that  $T$ is $\Psi$-outward satisfiable at $\kappa$. 
  \end{claim*}
  
  \begin{proof}[Proof of the Claim]
   Let $\lambda<\kappa$ be a cardinal, let $\vartheta>\eta$ be a cardinal with $T\in H_\vartheta$, 
   let $G$ be $\Coll(\omega,\vartheta)$-generic over $V$ and let $N\in V[G]$ be an outer $\ZFC^*$-model of $V_\vartheta^V$ with the property that $\Psi(N,V_\vartheta^V,\lambda,\kappa,\eta)$ holds in $V[G]$. 
  Since $N$ is a model of $\ZFC^*$ and $T\in H_\vartheta^V$ implies that $\theta<\vartheta$, we can apply \ref{Prop:F} in Definition \ref{definition:Closenessss}.\ref{item:Closeness} to infer that $\Psi(V_\theta^N,V_\theta^V,\lambda,\kappa,\eta)$ holds in $V[G]$. 
  Moreover, since $\Psi$ is a $\Delta_1^{\ZFC^*}$-formula  and all occurring parameters are elements of $N$, we know that $\Psi(V_\theta^N,V_\theta^V,\lambda,\kappa,\eta)$  also holds in $N$. 
   In addition, the  fact that $N$ is a model of $\ZFC^*$  ensures that $$\langle V_\theta^V,\in\rangle\models_{\LL^2}\Psi^*(\lambda,\kappa,\eta)$$ holds in $N$ and hence we know that $$\langle V_{\theta+1}^V,\in\rangle\models_{\LL^2}\Psi^\prime(\lambda,\kappa,\eta,V_\theta^V)$$ holds in $N$. 
   
   Now, let $T_0$ be a subtheory of $T$ in $N$ that has cardinality less than $\lambda$ in $N$.  
   Since $\lambda$ is a cardinal in $N$, we can  construct a model of $T_0$ with domain $V_{\theta+1}^V$ in $N$ that interprets $b$ as $\lambda$    and all constant symbols of the form $d_\gamma$ that appear in sentences in $T_0$ as ordinals less than $\lambda$. 
  \end{proof}

 Our setup now ensures that $T$ is satisfiable. Hence, we find a transitive set $M$ and an elementary embedding $\map{j}{V_{\theta+1}}{M}$  such that $j(\kappa)>\zeta$ and $$\langle M,\in\rangle\models_{\LL^2}\Psi^\prime(\zeta,j(\kappa),j(\eta),j(V_\theta)).$$  
Elementarity then implies that $M\cap\Ord=j(\theta)+1$ and $j(V_\theta)=M\cap V_{j(\theta)}$ is transitive. 
 In particular, we know that $$\langle M\cap V_{j(\theta)},\in\rangle\models_{\LL^2}\Psi^*(\zeta,j(\kappa),j(\eta)),$$ and therefore we can conclude that $\Psi(V_{j(\theta)},M\cap V_{j(\theta)},\zeta,j(\kappa),j(\eta))$ holds. 
 %
 These computations allow us to conclude that $\kappa$ is $\Psi$-large.  
\end{proof}

In combination with the results of the previous section, the above theorem now shows that all of the large cardinal properties that we considered so far can be characterized through compactness properties of second-order logic.

\begin{corollary}\label{corollary:IndividualChar}
 \begin{enumerate}
  \item A  cardinal $\kappa$ is a  $\Psi_{ms}$-outward compactness cardinal for $\LL^2$ if and only if there  is a measurable cardinal less than or equal to $\kappa$.  
  
  \item A cardinal $\kappa$ is a  $\Psi_{str}$-outward compactness cardinal for $\LL^2$ if and only if there  is a strong cardinal less than or equal to $\kappa$.  
    
   \item A cardinal $\kappa$ is a  $\Psi_{sc}$-outward compactness cardinal for $\LL^2$ if and only if there  is a supercompact cardinal less than or equal to $\kappa$. 
   
  \item\label{item:CorollaryCharOmega1SC} A cardinal $\kappa$ is a  $\Psi_{stc}$-outward compactness cardinal for $\LL^2$ if and only if $\kappa$ is $\omega_1$-strongly compact. 
  
  \item A  cardinal $\kappa$ is a  $\Psi_{ext}$-outward compactness cardinal for $\LL^2$ if and only if there  is an extendible cardinal less than or equal to $\kappa$.  \qed 
 \end{enumerate}
\end{corollary}

By Proposition \ref{proposition:OutwardLessCompact}, extendible cardinals are $\Psi$-outward compactness cardinals for $\LL^2$ whenever $\Psi$ is an $\LL^2$-measure of closeness. In particular, it is not possible to use Theorem \ref{theorem:LL2Duality} to characterize large cardinal notions stronger than extendibility. 
 Therefore, we may view $\Psi_{ext}$ as the strongest $\LL^2$-measure of closeness. This view is further supported by the following observation:

 \begin{proposition}\label{proposition:OutwardCompactExtendible}
  If $\kappa$ is  an infinite cardinal, then every ${<}\kappa$-satisfiable $\LL^2$-theory is $\Psi_{ext}$-outward satisfiable. 
 \end{proposition}

 \begin{proof}
   Let $T$ be a ${<}\kappa$-satisfiable $\LL^2$-theory. Pick a cardinal $\eta>\kappa$ with the property that $T\in  H_\eta=V_\eta$ and $V_\eta$ is sufficiently elementary in $V$. Fix a cardinal $\lambda<\kappa$ and a cardinal $\vartheta>\eta$, and let $G$ be $\Coll(\omega,\vartheta)$-generic over $V$. Pick an outer $\ZFC^*$-model $N\in V[G]$ of $V_\vartheta^V$ with the property that $\Psi_{ext}(N,V_\vartheta^V,\lambda,\kappa,\eta)$ holds in $V[G]$. 
   We then have $V_\eta^N=V_\eta^V$. Let $T_0$ be a subtheory of $T$ of cardinality less than $\lambda$ in $N$. Then, $T_0$ is contained in  $V$ and has cardinality less than $\lambda$ in $V$. Hence, we know that $T_0$ is satisfiable in $V$ and our setup ensures that $T_0$ is also satisfiable in $V_\eta^V$. Since $V_\eta^N=V_\eta^V$, this shows that $T_0$ is also satisfiable in $N$. These computations prove that $T$ is ${<}\lambda$-satisfiable in $N$. 
\end{proof}




\section{On the naturalness of $\LL^2$-characterizations}\label{section:Natural?}

 The results of the previous two sections naturally raise the question which other large cardinal notions between measurability and extendibility can be characterized through outward compactness properties of $\LL^2$. In particular, Corollary \ref{corollary:IndividualChar}.\ref{item:CorollaryCharOmega1SC} directly motivates the question whether strong compactness can be characterized in this way. 
 Below, we will show that this is indeed possible. But, we will also argue that the presented characterization lacks several desirable features that the characterizations listed in Corollary \ref{corollary:IndividualChar} possess.

 Given a first-order formula $\psi(v_0,v_1,v_2)$ in the language of set theory, 
  we let $\psi^\up(v_0,\ldots,v_4)$ denote the canonical first-order formula in the language of set theory with the property that $\psi^\up(N,M,\mu,\nu,\rho)$ holds if and only if the tuple $\langle N,M,\mu,\nu,\rho\rangle$ is suitable and either $N\cap V_\rho\subseteq M$  or $\psi(\mu,\nu,\rho)$ holds in $M\cap V_{\rho+\omega}$.\footnote{The motivation behind this definition is that the first disjunct helps to ensure that \ref{Prop:B} in Definition \ref{definition:Closenessss}.\ref{item:Closeness} is satisfied and } 

 \begin{proposition}
  Given a first-order formula $\psi(v_0,v_1,v_2)$ in the language of set theory, the formula $\psi^\up$ is an $\LL^2$-measure of closeness. 
 \end{proposition}
 
 \begin{proof}
  The above definition directly ensures that  $\psi^\up$ satisfies \ref{Prop:A} and \ref{Prop:B} in Definition \ref{definition:Closenessss}.\ref{item:Closeness}. 
  Fix transitive sets $M\subseteq N$ closed under basic set operations with $M\cap\Ord=N\cap\Ord\in\Lim$ and   limit ordinals $\mu<\nu<\rho<\theta\in M$. 
  First, assume that $\psi^\up(N,M,\mu,\nu,\rho)$ holds. In the first case, 
  if $N\cap V_\rho\subseteq M$ holds, then $N\cap V_\theta\cap V_\rho\subseteq M\cap V_\theta$  holds and it follows that $\psi^\up(N\cap V_\theta,M\cap V_\theta,\mu,\nu,\rho)$ is also true. 
  Next, assume that $\psi(\mu,\nu,\rho)$ holds in $V_{\rho+\omega}\cap M$. Since $\rho+\omega\leq\theta$ and $V_{\rho+\omega}\cap (M\cap V_\theta)=V_{\rho+\omega}\cap M$, we then know that $\psi^\up(N\cap V_\theta,M\cap V_\theta,\mu,\nu,\rho)$ also holds in this case. Since an analogous case distinction also shows that  $\psi^\up(N\cap V_\theta,M\cap V_\theta,\mu,\nu,\rho)$ implies that $\psi^\up(N,M,\mu,\nu,\rho)$ holds, we can now conclude that $\psi^\up$ satisfies \ref{Prop:F} in Definition \ref{definition:Closenessss}.\ref{item:Closeness} and hence we know that \eqref{item:L2measure1} of Definition \ref{definition:moc} holds.  
  By arguing as in the proof of Proposition \ref{proposition:FormulasAreL2measures}, we see that \eqref{item:L2measure2} of Definition \ref{definition:moc} holds. Finally, it is easy to find an $\LL^2$-formula $\varphi(v,W)$ in the language of set theory 
   such that $\ZFC^*$ proves that for every transitive set $D$ and  every ordinal $\alpha$ in $D$, the set $D\cap V_\alpha$ is the unique element $X$ of $\mathcal P(D)$ with the property that $\langle D,\in\rangle\models_{\LL^2}\varphi(\alpha,X)$ holds. Therefore, it follows that $\psi^\up$ also satisfies  \eqref{item:L2measure3} of Definition \ref{definition:moc}.  
 \end{proof}

We now show how the above concept can be used to obtain characterizations of strongly compact cardinals through compactness properties of $\LL^2$. 
 Let $\psi_{stc}(v_0,v_1,v_2)$ denote the canonical first-order formula in the language of set theory with the property that $\psi_{stc}(\mu,\nu,\rho)$ holds if and only if for some cardinal $\mu<\kappa\leq\nu$, there exists a fine, ${<}\kappa$-complete ultrafilter on $\pow_\kappa(\rho)$.

  \begin{proposition}\label{proposition:BadCharStronglyCompact}
   A cardinal $\kappa$ is $\psi_{stc}^\up$-large if and only if there is a strongly compact cardinal less than or equal to $\kappa$. In particular,  a cardinal $\kappa$ is a $\psi_{stc}^\up$-outward compactness cardinal for $\LL^2$ if and only if there is a strongly compact cardinal less than or equal to $\kappa$. 
 \end{proposition}
 
 \begin{proof}
  First, assume that $\mu$ is a strongly compact cardinal and $\kappa\geq\mu$ is a cardinal. Fix a limit ordinal $\eta>\kappa$ and a cardinal $\lambda\geq\eta$. 
  By {\cite[Corollary 22.18]{kanamori}}, there exists a fine, ${<}\mu$-complete ultrafilter on $\pow_\mu(\eta)$ and therefore $\psi_{stc}(\delta,\kappa,\eta)$ holds in $V_{\eta+\omega}$ for every cardinal $\delta<\mu$. 
  Since $\mu$ is measurable, we can find a transitive class $M$ and an elementary embedding $\map{j}{V}{M}$ with $j(\mu)>\lambda$. By elementarity, we know that $\psi_{stc}(\lambda,j(\kappa),j(\eta))$ holds in $V_{j(\eta)+\omega}^M$. This directly implies that $\psi_{str}^\up(V_\vartheta,V_\vartheta^M,\lambda,j(\kappa),j(\eta))$ holds for all limit ordinals $\vartheta>j(\eta)$. Using Lemma \ref{lemma:EmbeddingLarge}, we can now conclude that $\kappa$ is $\psi_{stc}^\up$-large.

  Now, assume, towards a contradiction, that $\kappa$ is a $\psi_{stc}^\up$-large cardinal and no cardinal less than or equal to $\kappa$ is strongly compact. Then there exists an ordinal $\alpha>\kappa$ with the property that for every cardinal $\mu\leq\kappa$, there is no fine, ${<}\mu$-complete  ultrafilter on $\pow_\mu(\alpha)$. 
  Fix a limit ordinal  $\eta>\alpha$. By Lemma \ref{lemma:EmbeddingLarge}, we can now find a cardinal $\lambda>\eta$, an inner model $M$ and an elementary embedding $\map{j}{V}{M}$ with $j(\kappa)>\lambda$ and the property that $\psi_{stc}^\up(V_{j(\eta)+\omega},V_{j(\eta)+\omega}^M,\lambda,j(\kappa),j(\eta))$ holds. 
  Since $\eta>\alpha$, we know that for all cardinals $\mu\leq\kappa$, there is no fine, ${<}\mu$-complete ultrafilter on $\pow_\mu(\eta)$. This shows that for all $\delta<\kappa$, the statement  $\psi_{stc}(\delta,\kappa,\eta)$ does not hold in $V_{\eta+\omega}$. By elementarity, $\psi_{stc}(\lambda,j(\kappa),j(\eta))$ fails in $V_{j(\eta)+\omega}^M$ and, by the definition of $\psi_{stc}^\up$, we thus know that $V_{j(\eta)+\omega}=V_{j(\eta)+\omega}^M$. 
  %
   %
   In particular, we have  $V_{j(\eta)}\subseteq M$. These computations show that for every limit ordinal $\eta>\kappa$, there is an inner model $M$ and an elementary embedding $\map{j}{V}{M}$ with $j(\kappa)>\eta$ and $V_{j(\eta)}\subseteq M$. This directly implies that $\kappa$ is $\Psi_{ext}$-large and hence Lemma \ref{lemma:ExtLarge} ensures that there is an extendible cardinal less than or equal to $\kappa$, contradicting our assumption, as extendibility implies strong compactness. 
 \end{proof}

 Note that the above arguments can be modified to obtain analogous characterizations for various types of large cardinal properties defined through the existence of certain filters. 
 Nevertheless, we consider the outward compactness characterizations obtained in this way as unsatisfying, 
 because 
  the measures of closeness $\psi^\up$ utilized in these characterizations  make no substantial demands on their first input parameter, and therefore   $\psi^\up$-outward consistency in $V$   is essentially decided by   properties of the ground model $V$. 
 In the remainder of this section, we aim to isolate a criterion for the naturalness of outward compactness characterizations that will allow us to separate the 
 characterizations obtained in Section \ref{section:CharLL2} from the 
 characterization of strong compactness derived above. 
   The formulation of this criterion is motivated by Propositions \ref{proposition:OutwardLessCompact} and \ref{proposition:OutwardCompactExtendible} which, in conjunction, show that ${<}\kappa$-satisfiability provably coincides with $\Psi_{ext}$-outward satisfiability at all limit cardinals $\kappa$. 
  Moreover, if $\Psi$ is an  $\LL^2$-measure of closeness and $\kappa$ is a $\Psi$-large cardinal that is not extendible, then Theorems \ref{theorem:Magidor} and \ref{theorem:LL2Duality} show that there is a ${<}\kappa$-satisfiable $\LL^2$-theory that is not $\Psi$-outward satisfiable at $\kappa$. 
  Therefore, it seems reasonable to expect that $\LL^2$-measures of closeness $\Psi$ that canonically correspond to large cardinal notions strictly weaker than extendibility have the property that $\Psi$-outwards satisfiability differs   from ${<}\kappa$-satisfiability at some cardinal $\kappa$.\footnote{Note that, by the above observation, the relevant case is when there are no $\Psi$-large cardinals. For a reasonable measure of closeness $\Psi$, we expect that  the relationship between the first and second input parameter of $\Psi$ should always be able to influence the truth of $\Psi$.} This intuition is captured in the following definition:

 \begin{definition}
  We say that a first-order formula $\Psi(v_0,\ldots,v_4)$ in the language of set theory \emph{naturally induces a large cardinal property below extendibility via outward compactness} if $\ZFC$ proves that for some infinite cardinal $\kappa$, there is a ${<}\kappa$-satisfiable $\LL^2$-theory that is not  $\Psi$-outward satisfiable at $\kappa$.   
 \end{definition}

 A short argument now shows that the formulae introduced in Section \ref{section:PsiLarge} that correspond to large cardinal notions below extendibility fulfill the above criterion:

\begin{proposition}
 If $\Psi$ is one of the formulae $\Psi_{ms}$, $\Psi_{str}$, $\Psi_{stc}$ or  $\Psi_{sc}$, then there is an  $\LL^2$-sentence $\varphi$  in the language of set theory such that $\{\varphi\}$ is a satisfiable $\LL^2$-theory that is not $\Psi$-outward satisfiable at any uncountable cardinal $\kappa$. 
\end{proposition}

 \begin{proof}
   There are $\LL^2$-sentences $\varphi_0$ and $\varphi_1$ in the language of set theory such that $\ZFC^*$ proves the following statements: 
    \begin{itemize}
     \item If $i<2$ and $\langle M,E\rangle\models_{\LL^2}\varphi_i$, then $E$ is a well-founded and extensional relation on $E$ and the corresponding  transitive collapse maps $M$ to $V_{\omega_2+1}$. 
     
     \item We have $$2^{\aleph_1}=\aleph_2 ~ \Longleftrightarrow ~ \langle V_{\omega_2+1},\in\rangle\models_{\LL^2}\varphi_0 ~ \Longleftrightarrow ~ \langle V_{\omega_2+1},\in\rangle\models_{\LL^2}\neg\varphi_1.$$
    \end{itemize}
    
   Now, fix $i<2$ with $\langle V_{\omega_2+1},\in\rangle\models_{\LL^2}\varphi_i$.  Note that there is a $\sigma$-closed partial order $\mathbb{P}$ with the property that whenever $G$ is $\mathbb{P}$-generic over $V$, then $\langle V_{\omega_2+1}^{V[G]},\in\rangle\models\neg\varphi_i$ holds in $V[G]$.

 Assume, towards a contradiction, that $\eta$ is a limit ordinal that witnesses that the theory $\{\varphi_i\}$ is $\Psi$-outward satisfiable at some uncountable cardinal $\kappa$.     Pick a cardinal $\vartheta>\eta$ such that $\mathbb{P}$ is an element of $H_\vartheta$ and $\mathbb{P}$ forces that  $V_\vartheta^{V[G]}$ is a model of $\ZFC^*$ whenever $G$ is $\mathbb{P}$-generic over $V$. Let $H$ be $\Coll(\omega,\vartheta)$-generic over $V$. 
 Using \cite[Proposition 10.20]{kanamori}, we can find $G\in V[H]$ that is $\mathbb{P}$-generic over $V$. Our setup now ensures that $V_\vartheta^{V[G]}$ is an outer $\ZFC^*$-model of $V_\vartheta^V$ in $V[H]$. Moreover, since $\mathbb{P}$ is $\sigma$-closed in $V$,  it is easy to check, in each case, that $\Psi(V_\vartheta^{V[G]},V_\vartheta^V,\omega,\kappa,\eta)$ holds in $V[H]$. It now follows that $\{\varphi_i\}$ is satisfiable in $V_\vartheta^{V[G]}$ and    this implies that $\langle V^{V[G]}_{\omega_2+1},\in\rangle\models_{\LL^2}\varphi_i$    holds in $V[G]$, a contradiction.  
 \end{proof}

 Note that the above proof does not work for $\psi_{stc}^\up$. Given an uncountable cardinal $\kappa$ with no strongly compact cardinals below it, we can pick a limit ordinal $\eta>\kappa+\omega_2$ with the property that for every  cardinal $\mu\leq\kappa$, there is no fine, ${<}\mu$-complete ultrafilter on $\pow_\mu(\eta)$, and forcing with the given partial order $\mathbb{P}$ adds an element of $V_\eta$. 
 Now, if $\lambda<\kappa$ is a cardinal, $\vartheta>\eta$ is a cardinal, $H$ is $\Coll(\omega,\vartheta)$-generic over $V$ and $N$ is an outer $\ZFC^*$-model of $V_\vartheta^V$ with the property that $\psi_{stc}^\up(N,V_\vartheta^V,\lambda,\kappa,\eta)$ holds in $V[H]$, then the  fact that $\psi_{stc}(\lambda,\kappa,\eta)$ does not hold in $V_{\eta+\omega}^V$ implies that $V_\eta^N=V_\eta^V$ and we can conclude that $\langle V_{\omega_2+1}^N,\in\rangle\models_{\LL^2}\varphi_i$ holds in $N$. This shows that the $\LL^2$-theory $\{\varphi_i\}$ is $\psi_{stc}^\up$-outward satisfiable at $\kappa$.

 \begin{corollary}
   The formulae $\Psi_{ms}$, $\Psi_{str}$, $\Psi_{stc}$ and  $\Psi_{sc}$ all naturally induce a large cardinal property below extendibility via outward compactness. \qed 
 \end{corollary}

 Now, we isolate a sufficient criterion for non-natural large cardinal characterizations.  We will show that $\psi_{stc}^\up$ satisfies this criterion below the first strongly compact cardinal.

 \begin{definition}
  Given a first-order formula $\Psi(v_0,\ldots,v_4)$ in the language of set theory, we say that \emph{$\Psi$-outward compactness trivializes at a cardinal $\kappa$} if there is a proper class of limit ordinals $\eta>\kappa$ with the property that for all cardinals $\lambda<\kappa$ and all limit ordinals $\vartheta>\eta$, the partial order $\Coll(\omega,\vartheta)$ forces that $N\cap V_\eta^{V[G]}\subseteq V$ holds whenever $G$ is $\Coll(\omega,\vartheta)$-generic over $V$ and $N\in V[G]$ is an outer $\ZFC^*$-model  of $V_\vartheta^V$ with the property that $\Psi(N,V_\vartheta^V,\lambda,\kappa,\eta)$ holds in $V[G]$. 
 \end{definition}

 As a first observation, we show that the above property causes $\Psi$-largeness to be equivalent to $\Psi_{ext}$-largeness:

 \begin{proposition}
  Let $\Psi(v_0,\ldots,v_4)$ be an $\LL^2$-measure of closeness. If $\Psi$-outward compactness trivializes at a $\Psi$-large cardinal $\kappa$, then there exists an extendible cardinal less than or equal to~$\kappa$. 
 \end{proposition}
 
 \begin{proof}
  Assume, towards a contradiction, that $\Psi$-outward compactness trivializes at a $\Psi$-large cardinal $\kappa$ and no cardinal  less than or equal to $\kappa$ is extendible. By {\cite[Proposition 23.15]{kanamori}}, we can find an ordinal $\alpha>\kappa$ with the property that for every ordinal $\beta$, there is no non-trivial elementary embedding $\map{j}{V_\alpha}{V_\beta}$ with $\crit(j)\leq\kappa$. 
  Since $\Psi$-outward compactness trivializes at $\kappa$, 
  there exists a limit ordinal $\eta>\alpha$ such that for all cardinals $\lambda<\kappa$ and all limit ordinals $\vartheta>\eta$, the partial order $\Coll(\omega,\vartheta)$ forces that $N\cap V_\eta^{V[G]}\subseteq V$ holds 
  whenever $G$ is $\Coll(\omega,\vartheta)$-generic over $V$ and $N\in V[G]$ is an outer $\ZFC^*$-model  of $V_\vartheta^V$ with the property that $\Psi(N,V_\vartheta^V,\lambda,\kappa,\eta)$ holds in $V[G]$. 
  Using Lemma \ref{lemma:EmbeddingLarge}, we now find a cardinal $\lambda>\eta$, an inner model $M$ and an elementary embedding $\map{j}{V}{M}$ such that $j(\kappa)>\lambda$ and  $\Psi(V_\vartheta,V_\vartheta^M,\lambda,j(\kappa),j(\eta))$ holds for every  limit ordinal $\vartheta>j(\eta)$. 
  Pick a limit ordinal $\vartheta>j(\eta)$ with the property that $j(\vartheta)=\vartheta$ and $V_\vartheta$ is a model of $\ZFC^*$. Let $G$ be $\Coll(\omega,\vartheta)$-generic over $V$. By using the elementarity of $j$ and Shoenfield absoluteness as in the proof of Theorem \ref{theorem:LL2Duality}, it now follows that, in $V[G]$, we have $N\cap V_{j(\eta)}^{V[G]}\subseteq V_\vartheta^M$ whenever $N$ is an outer $\ZFC^*$-model of $V_\vartheta^M$ with the property that $\Psi(N,V_\vartheta^M,\lambda,j(\kappa),j(\eta))$ holds.  
  Since $\alpha<j(\eta)$ and our assumptions on $\Psi$ ensure that $\Psi(V_\vartheta^V,V_\vartheta^M,\lambda,j(\kappa),j(\eta))$ holds in both $V$ and $V[G]$, we can now conclude that  $V_{j(\eta)}^V\subseteq M$ and $\map{j\restriction V_\alpha}{V_\alpha}{V_{j(\alpha)}}$ is a non-trivial elementary embedding with $\crit(j\restriction V_\alpha)\leq\kappa$, contradicting our initial assumption.  
  \end{proof}

 We now show how our two criteria are related:

 \begin{lemma}\label{lemma:CompactnessTrivial}
  Let $\Psi(v_0,\ldots,v_4)$ be a first-order formula in the language of set theory. If  $\Psi$-outward compactness trivializes at an infinite cardinal $\kappa$, then every ${<}\kappa$-satisfiable $\LL^2$-theory is $\Psi$-outward satisfiable at $\kappa$. 
 \end{lemma}
 
 \begin{proof}
  Let $T$ be a ${<}\kappa$-satisfiable $\LL^2$-theory. Then there exists a limit ordinal $\eta>\kappa$ such that the following statements hold: 
  \begin{itemize}
   \item $T\in H_\rho$ for some cardinal $\rho<\eta$. 
   
   \item Every subtheory of $T$ of cardinality less than $\kappa$ has a model in $V_\eta$. 
   
   \item For all cardinals $\lambda<\kappa$ and all limit ordinals $\vartheta>\eta$, the partial order $\Coll(\omega,\vartheta)$ forces that $N\cap V_{\eta+1}^{V[G]}\subseteq V$ holds for every outer $\ZFC^*$-model $N$ of $V_\vartheta^V$ with the property that $\Psi(N,V_\vartheta^V,\lambda,\kappa,\eta)$ holds. 
  \end{itemize}
  
    Now, let $\lambda<\kappa$ be a cardinal, let $\vartheta>\eta$ be a cardinal, let $G$ be $\Coll(\omega,\vartheta)$-generic over $V$ and let $N$ be an outer $\ZFC^*$-model of $V_\vartheta^V$ in $V[G]$ with the property that $\Psi(N,V_\vartheta^V,\lambda,\kappa,\eta)$ holds in $V[G]$. We then know that $N\cap V_\eta^{V[G]}\subseteq V$ and therefore every subtheory of $T$ of cardinality less than $\kappa$ in $N$ is contained in $V$ and has cardinality less than $\kappa$ in $V$. In particular, the theory $T$ is ${<}\lambda$-satisfiable in $N$.  
 \end{proof}

 \begin{corollary}\label{corollary:NaturalNotTrivialize}
  If  $\Psi(v_0,\ldots,v_4)$ is a first-order formula in the language of set theory that naturally induces a large cardinal property below extendibility via outward compactness, then $\ZFC$ proves that there is an infinite cardinal at which $\Psi$-outward compactness does not trivialize. \qed 
 \end{corollary}

  We are now ready to show that, in contrast to the characterizations derived in Section \ref{section:CharLL2}, the characterization of strong compactness given by Proposition \ref{proposition:BadCharStronglyCompact} does not meet the above criteria.

 \begin{proposition}\label{proposition:stcTrivial}
  If $\kappa$ is a cardinal with the property that no cardinal less than or equal to $\kappa$ is strongly compact, then $\psi_{stc}^\up$-outward compactness trivializes at $\kappa$. 
 \end{proposition}
 
 \begin{proof}
  Fix   a limit ordinal $\eta>\kappa$ with the property that for every infinite cardinal $\mu\leq\kappa$, there is no fine, ${<}\mu$-complete ultrafilter on $\pow_\mu(\eta)$. Let $\lambda<\kappa$ be a cardinal, let $\vartheta>\kappa$ be a cardinal, let $G$ be $\Coll(\omega,\vartheta)$-generic over $V$ and let $N$ be  an outer $\ZFC^*$-model of $V_\vartheta^V$ in $V[G]$ with the property that $\psi_{stc}^\up(N,V_\vartheta^V,\lambda,\kappa,\eta)$ holds in $V[G]$. 
 Since $\psi_{stc}(\lambda,\kappa,\eta)$ does not hold in $V_{\eta+\omega}^V$, we can now conclude that 
   $N\cap V_\eta^{V[G]}\subseteq V$ holds. 
 \end{proof}

 \begin{corollary}
  If the formula $\psi_{stc}^\up$ naturally induces a large cardinal property below extendibility via outward compactness, then  $\ZFC$ is inconsistent.  
 \end{corollary}
 
 \begin{proof}
  Assume that the formula $\psi_{stc}^\up$ naturally induces a large cardinal property below extendibility via outward compactness. Then Corollary \ref{corollary:NaturalNotTrivialize} shows that $\ZFC$ proves that there is an infinite cardinal at which $\Psi$-outward compactness does not trivialize. 
  By Proposition \ref{proposition:stcTrivial}, this implies that $\ZFC$ proves the existence of a strongly compact cardinal and therefore we can conclude that $\ZFC$ is inconsistent.  
 \end{proof}


\section{Fragments of Vop\v{e}nka's Principle}\label{section:FragmentsVP}

 We now aim to extend the equivalence established in Theorem \ref{theorem:LL2Duality} to obtain characterizations of stronger large cardinal assumptions. 
 For this purpose, we start by introducing a natural strengthening of the the notion of $\Psi$-largeness that is motivated by classical characterizations of Woodin cardinals (see {\cite[Theorem 26.14]{kanamori}}) and related variations of supercompactness and extendibility (see {\cite[p. 339]{kanamori}}).

 \begin{definition}\label{definition:PsiLargeA}
 Given a formula  $\Psi(v_0,\ldots,v_4)$  in the language of set theory and a class $A$,\footnote{Note that, since we are working in a $\ZFC$ context, all classes are definable.} 
 a cardinal $\kappa$ is \emph{$\Psi$-large for $A$} if for all limit ordinals $\kappa<\eta<\theta$, 
 there are unboundedly many cardinals $\lambda\geq\eta$ with the property that there is a transitive set $M$ and an elementary embedding $\map{j}{V_{\theta+1}}{M}$ 
 such that $j(\kappa)>\lambda$, $$j(A\cap V_\kappa)\cap V_\lambda ~ = ~ A\cap M\cap  V_\lambda$$ and $\Psi(V_{j(\theta)},M\cap V_{j(\theta)},\lambda,j(\kappa),j(\eta))$ holds.  
\end{definition}

 In the following, we present the first example of a large cardinal principle that we aim to represent through the above concept. 
 Remember that \emph{Vop\v{e}nka's Principle} is a scheme of axioms stating that for every proper class $\mathcal{C}$ of structures of the same type, there exist $A,B\in\mathcal{C}$ with $A\neq B$ and an elementary embedding from $A$ to $B$. 
  Using results from \cite{Cn} and \cite{MR3323199}, it is easy to show that Vop\v{e}nka's Principle provides various examples of cardinals with the given property.

\begin{proposition}\label{proposition:VPMaximalLargeness}
 If Vop\v{e}nka's Principle holds, then for every measure of closeness $\Psi(v_0,\ldots,v_4)$ and every class $A$, there is a proper class of cardinals that are $\Psi$-large for $A$. 
\end{proposition}

\begin{proof}
 Pick a natural number $n>0$, a $\Sigma_n$-formula $\varphi(v_0,v_1)$, and a set $y$ with the property that $A=\Set{x}{\varphi(x,y)}$, and pick an ordinal $\xi$. 
 Using {\cite[Corollary 6.9]{MR3323199}}, we find a cardinal $\kappa>\xi$, with $y\in V_\kappa$, that is \emph{$C^{(n)}$-extendible}, {i.e.,} the cardinal $\kappa$ has the property that for every $\lambda>\kappa$ with $V_\lambda\prec_{\Sigma_n}V$, there exists an ordinal $\zeta$ and an elementary embedding $\map{j}{V_\lambda}{V_\zeta}$ with $\crit(j)=\kappa$, $j(\kappa)>\lambda$ and $V_{j(\kappa)}\prec_{\Sigma_n}V$ (see {\cite[Definition 3.2]{Cn}}). 
 We can now use {\cite[Proposition 3.4]{MR3323199}} to see that $V_\kappa\prec_{\Sigma_{n+2}}V$ and hence $A\cap V_\kappa=\Set{x\in V_\kappa}{V_\kappa\models\varphi(x,y)}$. In the following, fix limit ordinals $\theta>\eta>\kappa$ and a cardinal $\lambda>\theta$ with $V_\lambda\prec_{\Sigma_n}V$. In this situation, we can  find an ordinal $\zeta$ and an elementary embedding $\map{j}{V_\lambda}{V_\zeta}$ with $\crit(j)=\kappa$, $j(\kappa)>\lambda$ and $V_{j(\kappa)}\prec_{\Sigma_n}V$. Elementarity then implies that $$j(A\cap V_\kappa)\cap V_\lambda ~ = ~ \Set{x\in V_\lambda}{V_{j(\kappa)}\models\varphi(x,y)} ~ = ~ A\cap V_\lambda.$$ 
 Since \ref{Prop:B} in Definition \ref{definition:Closenessss}.\ref{item:Closeness} ensures that $\Psi(V_{j(\theta)},V_{j(\theta)},\lambda,j(\kappa),j(\eta))$ holds, we can conclude that the embedding $\map{j\restriction V_{\theta+1}}{V_{\theta+1}}{V_{j(\theta)+1}}$ possesses the desired properties with respect to $\eta$ and~$\lambda$. 
\end{proof}

Next, we prove the direct analog of Lemma \ref{lemma:EmbeddingLarge} with respect to Definition \ref{definition:PsiLargeA}.

\begin{lemma}\label{lemma:EmbeddingLargeForA}
 Given a measure of closeness $\Psi(v_0,\ldots,v_4)$ and a class $A$, a  cardinal $\kappa$ is $\Psi$-large for $A$ if and only if for every limit ordinal $\eta>\kappa$, there exist unboundedly many cardinals $\lambda\geq\eta$ with the property that there is an inner model $M$ and an elementary embedding $\map{j}{V}{M}$ such that $j(\kappa)>\lambda$, $j(A\cap V_\kappa)\cap V_\lambda=A\cap M\cap V_\lambda$ and $\Psi(V_\vartheta,V_\vartheta^M,\lambda,j(\kappa),j(\eta))$ holds for all limit ordinals $\vartheta>j(\eta)$. 
\end{lemma}

\begin{proof}
  First, assume that $\kappa$ is $\Psi$-large for $A$ and $\zeta>\eta>\kappa$ are limit ordinals. 
 We find  a sufficiently large fixed point $\nu>\zeta$ of the $\beth$-function, a cardinal $\lambda\geq\zeta$,  a transitive set $M_0$ and an elementary embedding $\map{i}{V_{\nu+1}}{M_0}$ such that  $i(\kappa)>\lambda$, $i(A\cap V_\kappa)\cap V_\lambda=A\cap M\cap V_\lambda$  and $\Psi(V_{i(\nu)},M_0\cap V_{i(\nu)},\lambda,i(\kappa),i(\eta))$ holds. 
 By repeating the argument from the proof of Lemma \ref{lemma:EmbeddingLarge}, we can construct an inner model $M$ with $V_{i(\nu)}^M=M_0\cap V_{i(\nu)}$ and an elementary embedding $\map{j}{V}{M}$ with $j\restriction(V_\nu\cup\{\nu\}) =  i\restriction(V_\nu\cup\{\nu\})$.  In particular, we know that $j(\kappa)>\lambda$ and $j(A\cap V_\kappa)\cap V_\lambda=A\cap M\cap V_\lambda$. 
Moreover, for every  limit ordinal $\vartheta\geq j(\eta)$,  we can  use \ref{Prop:F} in Definition \ref{definition:Closenessss}.\ref{item:Closeness} to show that $\Psi(V_\vartheta,V_\vartheta^M,\lambda,j(\kappa),j(\eta))$ holds.

 Now, assume that for every limit ordinal $\eta>\kappa$, there exist unboundedly many  cardinals $\lambda\geq\eta$ with the property that there is a transitive class $M$ and an elementary embedding $\map{j}{V}{M}$ with $j(\kappa)>\lambda$, $j(A\cap V_\kappa)\cap V_\lambda=A\cap V_\lambda$ and the property that $\Psi(V_\vartheta,V_\vartheta^M,\lambda,j(\kappa),j(\eta))$ holds for all  limit ordinals $\vartheta>j(\eta)$. 
 Pick limit ordinals $\theta>\eta>\kappa$ and an ordinal $\zeta>\eta$. Then there is a cardinal $\lambda>\zeta$, an inner model $M$ and an elementary embedding $\map{j}{V}{M}$ such that the above statements hold with respect to  $\kappa$, $\eta$ and $\lambda$. 
 We then know that $M\cap V_{j(\theta)+1}$ is a transitive set and the map $\map{j\restriction V_{\theta+1}}{V_{\theta+1}}{M\cap V_{j(\theta)+1}}$ possesses all of the desired properties.  
\end{proof}


\subsection{Vop\v{e}nka's Principle}
 We now show how the validity of Vop\v{e}nka's Principle can be characterized through the existence of  cardinals that are $\Psi$-large for certain classes. 
  Note that the results of {\cite[Section 4]{Cn}} show that, over $\ZFC$, Vop\v{e}nka's Principle is equivalent to the scheme that only requires the instances of Vop\v{e}nka's Principle for classes of finite signatures that are defined by formulas without parameters to hold. 
 The following lemma is a direct adaptation of the results discussed in {\cite[p. 339]{kanamori}} to our setting.

\begin{lemma}\label{lemma:LargeVP}
 The following schemes are equivalent over $\ZFC$: 
 \begin{enumerate}
  \item\label{item:VP1} Vop\v{e}nka's Principle. 
  
  \item\label{item:VP2} For every class $A$, there is a proper class of cardinals that are $\Psi_{ext}$-large for $A$. 
  
  \item\label{item:VP3} For every parameter-free formula $\varphi(v)$ in the language of set theory, there is a cardinal that is $\Psi_{sc}$-large for the class $\Set{x}{\varphi(x)}$.  
 \end{enumerate}
\end{lemma}

\begin{proof}
  First, note that the implication from \eqref{item:VP1} to \eqref{item:VP2} already follows from Proposition \ref{proposition:VPMaximalLargeness}, and the  implication from \eqref{item:VP2} to \eqref{item:VP3}  is trivial. 
  Now, assume that \eqref{item:VP3} holds.  Fix a proper class $\mathcal{C}$ of structures of the same finite signature that is definable by a formula $\varphi(v)$ without parameters. 
  Let $\kappa$ be a cardinal that is $\Psi_{sc}$-large for $\mathcal{C}$. By our assumptions, there is a structure $B$ in $\mathcal{C}$ that is not an element of $H_\kappa$. Fix a cardinal $\eta>\kappa$ with $B\in H_\eta=V_\eta$ and a cardinal $\theta>\eta$ such that the formula $\varphi(v)$ is absolute between $V_\theta$ and $V$. 
  In this situation, we can find a cardinal $\lambda\geq\eta$, a transitive set $M$ and an elementary embedding $\map{j}{V_{\theta+1}}{M}$ such that $j(\kappa)>\lambda$, $j(\mathcal{C}\cap V_\kappa)\cap V_\lambda=\mathcal{C}\cap V_\lambda$, and $\Psi_{sc}(V_{j(\theta)},M\cap V_{j(\theta)},\lambda,j(\kappa),j(\eta))$ holds. 
  We then know that $B\in H_{j(\kappa)}^M$ and hence elementarity implies that $j(B)\neq B$. 
 This shows that $j(B)$  is a structure of the given signature and the embedding $j$ induces an elementary embedding of $B$ into $j(B)$. Moreover, our setup ensures that the set $M$ contains both $B$ and the given elementary embedding of $B$ into $j(B)$. 
  In addition, we know that $$B ~ \in ~ \mathcal{C}\cap V_\lambda ~ \subseteq ~ j(\mathcal{C}\cap V_\kappa) ~ = ~ j(\Set{x\in V_\kappa}{V_\theta\models\varphi(x)}).$$ By elementarity, this implies that, in $V_{\theta+1}$, there is an elementary embedding of a structure $A$  of the given signature into $B$ such that $A\neq B$ and $V_\theta\models\varphi(A)$. This shows that, in $V$, there is a structure $A\in\mathcal{C}$ with $A\neq B$ and an elementary embedding of $A$ into $B$.    As discussed earlier, the results of {\cite[Section 4]{Cn}} show that these computations ensure that Vop\v{e}nka's Principle holds.  
\end{proof}


\subsection{Ord is Woodin}
 Remember that \anf{\emph{$\Ord$ is Woodin}} is the scheme of axioms stating that for every class $A$, there exists a cardinal $\kappa$ that is \emph{strong for $A$}, {i.e.,} for every ordinal $\lambda$, $\kappa$ is \emph{$\lambda$-strong for $A$}: there is an inner model $M$ with $V_\lambda\subseteq M$ and an elementary embedding $\map{j}{V}{M}$ with $\crit(j)=\kappa$, $j(\kappa)>\lambda$ and $A\cap V_\lambda= j(A\cap V_\kappa)\cap V_\lambda$ (see \cite{bagaria_wilson_2022} and \cite{MR4439579}). 
  %
 The following lemma shows how this principle is related to the concepts introduced above. 
 Its proof is based on the \emph{$\Sigma_n$-Product Reflection Principle} introduced by Bagaria and Wilson in \cite{bagaria_wilson_2022}.

 Remember that  for any set $S$ of relational structures of the same type, the \emph{set-theoretic product $\prod S$} is defined to be the structure whose domain is the set of all functions $f$ with domain $S$ and the property that $f(\mathcal{A})\in A$ holds for every structure $\mathcal{A}=\langle A,\ldots\rangle$ in $S$, and whose interpretation of relation symbols is defined pointwise, {i.e.,} for every $(n+1)$-ary relation symbol $R$ in the given language, we have  $R^{\prod S}(f_0,\ldots,f_n)$  if and only if $R^{\mathcal{A}}(f_0(\mathcal{A}),\ldots,f_n(\mathcal{A}))$ holds for all $\mathcal{A}\in S$. 
 For a class $\mathcal{C}$ of relational structures of the same type, the principle $\mathrm{PRP}_{\mathcal{C}}$  states  that there exists a non-empty subset $S$ of $\mathcal{C}$ with the property that for every $\mathcal{A}$ in $\mathcal{C}$, there exists a homomorphism\footnote{As defined in {\cite[Section 1.2]{MR1221741}}.} of $\prod S$ into $\mathcal{A}$. 
 Given a natural number $n>0$, Bagaria and Wilson then defined the  \emph{$\Sigma_n$-Product Reflection Principle} to be the statement that $\mathrm{PRP}_{\mathcal{C}}$ holds for every non-empty class $\mathcal{C}$ of graphs that is definable by a $\Sigma_n$-formula without parameters (see {\cite[Definition 3.1]{bagaria_wilson_2022}}). 
 The results of  {\cite[Section 5]{bagaria_wilson_2022}} show that $\Ord$ is Woodin if and only if  the $\Sigma_n$-Product Reflection Principle holds for all $n>0$.

\begin{lemma}\label{lemma:OrdWoodinEqui}
 The following schemes are equivalent over $\ZFC$: 
 \begin{enumerate}
  \item\label{item:OrdWoodin1} $\Ord$ is Woodin. 

  \item\label{item:OrdWoodin2} For every class $A$, there is a proper class of cardinals that are $\Psi_{str}$-large for $A$.   
    
  \item\label{item:OrdWoodin3} For every parameter-free formula $\varphi(v)$ in the language of set theory, there is a cardinal that is $\Psi_{str}$-large for the class $\Set{x}{\varphi(x)}$.   
 \end{enumerate}
\end{lemma}

\begin{proof}
 First, assume that \eqref{item:OrdWoodin1} holds. Fix an ordinal $\xi$ and pick  a formula $\varphi(v_0,\ldots,v_m)$ and parameters $y_0,\ldots,y_{m-1}$ with $A=\Set{x}{\varphi(x,y_0,\ldots,y_{m-1})}$. Define $$A_* ~ = ~ \Set{\langle x_0,\ldots,x_m\rangle}{\varphi(x_0,\ldots,x_m)}.$$ Using {\cite[Theorem 5.13]{bagaria_wilson_2022}}, we find a cardinal $\kappa>\xi$ that satisfies $y_0,\ldots,y_{m-1}\in V_\kappa$ and   is strong for $A_*$. 
 Fix a cardinal $\lambda>\kappa$. We can then find an inner model $M$ with $V_\lambda\subseteq M$ and an elementary embedding $\map{j}{V}{M}$  with  $\crit(j)=\kappa$, $j(\kappa)>\lambda$ and $A_*\cap V_\lambda= j(A_*\cap V_\kappa)\cap V_\lambda$. 
 Given $x\in V_\lambda\subseteq M$, we then know that 
  \begin{equation*}
   \begin{split}
   x\in A ~ & \Longleftrightarrow ~ \langle x,y_0,\ldots,y_{m-1}\rangle\in A_* ~ \Longleftrightarrow ~  \langle x,y_0,\ldots,y_{m-1}\rangle\in j(A_*\cap V_\kappa) \\ 
     & \Longleftrightarrow ~ M\models\varphi(x,y_0,\ldots,y_{m-1}) ~ \Longleftrightarrow ~ x\in j(A\cap V_\kappa). 
    \end{split}
   \end{equation*}
  These equivalences show that $j(A\cap V_\kappa)\cap V_\lambda=A\cap V_\lambda$ holds. Moreover, it follows that $\Psi_{str}(V_\vartheta,V_\vartheta^M,\lambda,j(\kappa),j(\eta))$ holds for all limit ordinals $\kappa<\eta\leq\lambda$ and all  limit ordinals $\vartheta>j(\eta)$. 
  By Lemma \ref{lemma:EmbeddingLargeForA}, this shows that the cardinal $\kappa>\xi$ is $\Psi_{str}$-large for $A$. Hence, we have shown that \eqref{item:OrdWoodin2} holds in this case.

The implication from \eqref{item:OrdWoodin2} to \eqref{item:OrdWoodin3} is immediate.

 Now, assume that \eqref{item:OrdWoodin3} holds. 
 Let $\mathcal{C}$ denote a non-empty class of graphs that is defined by a formula $\varphi(v)$ without parameters. By our assumptions, there is  a cardinal $\kappa$ that  is $\Psi_{str}$-large for $\mathcal{C}$. We then know that $\mathcal{C}\cap V_\kappa\neq\emptyset$. Fix a graph $\mathcal{G}$ in $\mathcal{C}$ and pick a cardinal $\eta>\kappa$ with $\mathcal{G}\in H_\eta=V_\eta$. 
 Using Lemma \ref{lemma:EmbeddingLargeForA}, we now find a cardinal $\lambda>\eta$, a transitive class $M$ with $V_\lambda\subseteq M$ and an elementary embedding $\map{j}{V}{M}$ satisfying $j(\kappa)>\lambda$ and $j(\mathcal{C}\cap V_\kappa)\cap V_\lambda=\mathcal{C}\cap V_\lambda$. 
 In this situation,  we  know that $\mathcal{G}\in j(\mathcal{C}\cap V_\kappa)$. Since the graph  $j(\prod(\mathcal{C}\cap V_\kappa))$ is a subgraph of the product graph $\prod(j(\mathcal{C}\cap V_\kappa))$ and therefore the embedding $j$ induces a homomorphism of $\prod(\mathcal{C}\cap V_\kappa)$ into $\prod(j(\mathcal{C}\cap V_\kappa))$, we can combine this homomorphism with the projection from $\prod(j(\mathcal{C}\cap V_\kappa))$ to $\mathcal{G}$ to find a homomorphism from $\prod(\mathcal{C}\cap V_\kappa)$ to $\mathcal{G}$. These computations show that the $\Sigma_n$-Product Reflection Principle holds  for every natural number $n>0$ and therefore the results of {\cite[Section 5]{bagaria_wilson_2022}} allow us to conclude that \eqref{item:OrdWoodin1} holds in this case. 
\end{proof}


\section{Abstract logics}\label{section:AbstractLogics}

In order to connect the existence of cardinals that are $\Psi$-large for certain classes to compactness properties of strong logics, we now fix a notion of \emph{abstract logic}. 
Our setup is based on the definitions given in \cite[Section 2.5]{MR409165} and our exact presentation closely follows the one provided in \cite{bdgm}. 
%
In \cite{MR780522}, Makowsky  proved that Vop\v{e}nka's Principle is equivalent to the assumption that every abstract logic has a strong compactness cardinal (see also {\cite[Section 4]{bdgm}} and Theorem \ref{theorem:Makowsky} below), {i.e.,} that a direct analogue of Theorem \ref{theorem:Magidor} that replaces second-order logic with arbitrary abstract logics and the existence of an extendible cardinal with the validity of Vop\v{e}nka's Principle holds.  
 The goal of this section is to formulate and verify similar analogues of Theorem \ref{theorem:LL2Duality} and Corollary \ref{corollary:IndividualChar} for abstract logics.


\begin{definition}\label{definition:StructuresAndRenamings}
  \begin{enumerate}
    \item A \emph{language}  is a tuple $$\tau ~ = ~ \langle \mathfrak{F}_\tau,\mathfrak{R}_\tau,\mathfrak{C}_\tau,a_\tau\rangle$$ consisting of disjoint sets $\mathfrak{F}_\tau$ (the \emph{function symbols} in $\tau$), $\mathfrak{R}_\tau$ (the \emph{relation symbols} in $\tau$) and $\mathfrak{C}_\tau$ (the \emph{constant symbols} in $\tau$) together with a function $\map{a_\tau}{\mathfrak{F}_\tau\cup\mathfrak{R}_\tau}{\omega\setminus\{0\}}$ (the \emph{arity function} for $\tau$).

    \item Given a language $\tau$, a \emph{$\tau$-structure} is a tuple $$M ~ = ~ \langle \vert M\vert, \seq{f^M}{f\in\mathfrak{F}_\tau},\seq{r^M}{r\in\mathfrak{R}_\tau},\seq{c^M}{c\in\mathfrak{C}_\tau}\rangle,$$  where $\vert M\vert$ is a non-empty set (the \emph{universe of $M$}), each $f^M$ is an $a_\tau(f)$-ary function on $\vert M\vert$, each $r^M$ is an $a_\tau(r)$-ary relation on $\vert M\vert$, and each $c^M$ is an element of $\vert M\vert$. We let $\Str_\tau$ denote the collection of all $\tau$-structures.

    \item A \emph{renaming} of a language $\sigma$ into a language $\tau$ is a bijection $$\map{r}{\mathfrak{F}_\sigma\cup\mathfrak{R}_\sigma\cup\mathfrak{C}_\sigma}{\mathfrak{F}_\tau\cup\mathfrak{R}_\tau\cup\mathfrak{C}_\tau}$$ satisfying $r[\mathfrak{F}_\sigma]=\mathfrak{F}_\tau$, $r[\mathfrak{R}_\sigma]=\mathfrak{R}_\tau$, $r[\mathfrak{C}_\sigma]=\mathfrak{C}_\tau$ and $a_\sigma(s)=a_\tau(r(s))$ for all $s\in\mathfrak{F}_\sigma\cup\mathfrak{R}_\sigma$.  In this setting, we let $\map{r^*}{\Str_\sigma}{\Str_\tau}$ denote the canonical bijective class function induced by $r$ that fixes the universes of the given $\sigma$-structures. 
  \end{enumerate}
 \end{definition}

 Given the above definitions, there is a canonical notion of a \emph{sublanguage} of a given language. Moreover, given a sublanguage $\sigma$ of a language $\tau$, there is a canonical notion of a $\sigma$-reduct $M\restriction\sigma\in\Str_\sigma$ of a $\tau$-structure $M$. 
 Finally, given a language $\tau$, there is a canonical notion of an isomorphism of $\tau$-structures.

 \begin{definition}\label{definition:AbstractLogic}
  An \emph{abstract logic} $\langle\LL,\models_{\LL}\rangle$ consists of a class function $\LL$ and a binary class relation $\models_\LL$ that satisfy the following conditions: 
     \begin{enumerate}
      \item The domain of $\LL$ is the class of all languages. Given  a language $\tau$,  we call $\LL(\tau)$  the set of \emph{$\tau$-sentences} with respect to $\langle\LL,\models_{\LL}\rangle$. 
      
      \item If $M\models_\LL\phi$, then there is a language $\tau$ with $M\in\Str_\tau$ and $\phi\in\LL(\tau)$.  We call  $\models_{\LL}$ the \emph{satisfaction relation} of $\langle\LL,\models_{\LL}\rangle$. 
      
      \item  (Monotonicity) If $\sigma$ is a sublanguage of $\tau$, then $\LL(\sigma)\subseteq\LL(\tau)$. 
      
      \item (Expansion) If $\sigma$ is a sublanguage of $\tau$, $M\in\Str_\tau$ and $\phi\in\LL(\sigma)$, then $$M\models_\LL\phi ~ \Longleftrightarrow ~ {M\restriction\sigma}\models_\LL\phi.$$
      
       \item (Isomorphism) Given a language $\tau$ and isomorphic $\tau$-structures $M$ and $N$, we have $$M\models_\LL\phi ~ \Longleftrightarrow ~ N\models_\LL\phi$$ for all $\phi\in\LL(\tau)$.  
       
      \item\label{item:renaming}  (Renaming) Every renaming $r$ of a language $\sigma$ into a language $\tau$ induces a unique bijection $\map{r_*}{\LL(\sigma)}{\LL(\tau)}$ with the property that $$M\models_\LL\phi ~ \Longleftrightarrow ~ r^*(M)\models_\LL r_*(\phi)$$  holds for every $M\in\Str_\sigma$ and every $\phi\in\LL(\sigma)$.  
            
     \item\label{item:occurrenceNumber} There exists a cardinal $\ooo$   with the property that for every language $\tau$ and every $\phi\in\LL(\tau)$, there exists a sublanguage $\sigma$ of $\tau$ with $\vert\mathfrak{F}_\sigma\cup\mathfrak{R}_\sigma\cup\mathfrak{C}_\sigma\vert<\ooo$ and $\phi\in\LL(\sigma)$. 
     \end{enumerate}
  \end{definition}

  In the remainder of this section, we will usually  write $\LL$ instead of    $\langle\LL,\models_{\LL}\rangle$ to denote an abstract logic. 
 We call the least cardinal $\ooo$ witnessing \eqref{item:occurrenceNumber} in Definition \ref{definition:AbstractLogic} the \emph{occurrence number} of $\LL$. 
 We say that a set $T$ is an \emph{$\LL$-theory} if there is a language $\tau$ with $T\subseteq\LL(\tau)$. We say that an $\LL$-theory $T$ is \emph{satisfiable} if there is a language $\tau$ with $T\subseteq\LL(\tau)$ and  a $\tau$-structure $M$ with $M\models_\LL\phi$ for all $\phi\in T$. Given an infinite cardinal $\kappa$, we say that an $\LL$-theory $T$ is \emph{${<}\kappa$-satisfiable} if every subset of $T$ of cardinality less than $\kappa$ is satisfiable. 
 Finally, we say that an infinite cardinal $\kappa$ is a \emph{strong compactness cardinal for $\LL$} if every ${<}\kappa$-satisfiable $\LL$-theory is satisfiable. 
 Makowsky's result from \cite{MR780522}  now states the following:

 \begin{theorem}[Makowsky \cite{MR780522}]\label{theorem:Makowsky}
  The following schemes are equivalent over $\ZFC$: 
   \begin{enumerate}
   \item Vop\v{e}nka's Principle. 
   
   \item Every abstract logic has a strong compactness cardinal. 
   \end{enumerate}
 \end{theorem}

\begin{remark}\label{remark:UniquePermutation}
 A consequence of the way  Definition \ref{definition:AbstractLogic} is formulated, which will play a crucial role in subsequent proofs, is that Clause \eqref{item:renaming} of the definition ensures that 
 for every language $\tau$, 
 the identity map $\id_{\LL(\tau)}$ on $\LL(\tau)$ is the unique  permutation $\pi$ of $\LL(\tau)$ corresponding to the trivial renaming $\id_\tau$ of $\tau$ into itself, {i.e.,} we have $(\id_\tau)_*=\id_{\LL(\tau)}$. 
 This uniqueness then directly implies that for every language $\tau$ and every non-trivial\footnote{{I.e.,} not equal to the identity on $\LL(\tau)$.} permutation $\pi$ of $\LL(\tau)$, there exists $\chi\in\LL(\tau)$ and a $\tau$-structure $O$ with $$O\models_\LL\chi ~ \Longleftrightarrow ~ O \mathbin{{\not\models}_\LL} \pi(\chi).$$ 
 
 These observations show that when we want to view first-order logic and its extensions as abstract logics, then the elements of sets of the form $\LL(\tau)$ do not correspond to individual sentences of the given logic, but to equivalence classes of semantically equivalent sentences, because otherwise a transposition of $\LL(\tau)$ that exchanges two semantically equivalent sentences would contradict the uniqueness of $\id_{\LL(\tau)}$.  
\end{remark}

We now generalize the notions introduced in Section \ref{section:CharLL2} to obtain characterizations of the validity of fragments of Vop\v{e}nka's Principle through compactness properties of abstract logics. 
 In the following, we say that an abstract logic $\LL$ is defined by formulas  $\varphi_0(v_0,v_1,v_2)$ and $\varphi_1(v_0,v_1,v_2)$ together with a parameter $z$ if $\LL=\Set{\langle x,y\rangle}{\varphi_0(x,y,z)}$ and $\models_\LL=\Set{\langle x,y\rangle}{\varphi_1(x,y,z)}$.

\begin{definition}\label{definition:FPsiOutward}
Let $F$ be a theory in the language of set theory, let $\Psi(v_0,\ldots,v_4)$ be a formula\footnote{In most applications of this concept below, the formula $\Psi$ is assumed to be a $\Delta_1^{\ZFC^-}$-formula. It should be noted that, besides the absoluteness assumptions stated in Clause \eqref{item:FPsiOutwardSatisfiable1b} of his definition, this complexity assumption on $\Psi$ is essential for the proofs of the main results of this section, because it allows us to use \emph{Shoenfield Absoluteness} in a key step of our proofs. Moreover, Corollaries \ref{corollary:OutwardVP} and \ref{corollary:OutwardOrdWoodin} below show that the resulting outward compactness properties highly depend on the specific formula used. } in the language of set theory,  let $\kappa$ be an infinite cardinal and let $\LL$ be an abstract logic that is defined by formulas $\varphi_0$ and $\varphi_1$ together with a parameter $z$. 
 \begin{enumerate}
  \item\label{item:FPsiOutwardSatisfiable1}  An $\LL$-theory $T$ is \emph{$F$-$\Psi$-outward satisfiable at $\kappa$} if there is a limit ordinal $\eta>\kappa$ such that 
  for all cardinals $\lambda<\kappa$ and 
   all cardinals $\vartheta>\eta$ with $T,z\in H_\vartheta$,  
   the theory $T$ is a ${<}\lambda$-satisfiable $\LL^\prime$-theory in $N$, whenever $G$ is  $\Coll(\omega,\vartheta)$-generic over $V$,  $N\in V[G]$ is an outer $F$-model  of $V_\vartheta^V$,  $\LL^\prime$ is an abstract logic in $N$ and the following statements hold: 
  \begin{enumerate}
    \item $\Psi(N,V_\vartheta^V,\lambda,\kappa,\eta)$ holds in $V[G]$. 
   
    \item\label{item:FPsiOutwardSatisfiable1b} The formulas $\varphi_0$ and $\varphi_1$ are absolute from $V$ to $N$ with respect to parameters contained in $V_\lambda^V$.\footnote{{I.e.,} the statement $\varphi_i(x_0,x_1,x_2)$ holds in $N$ for all  $i<2$ and $x_0,x_1,x_2\in V_\lambda^V$ with the property that $\varphi_i(x_0,x_1,x_2)$ holds in $V$.}  
    
    \item The formulas  $\varphi_0$ and $\varphi_1$ together with the parameter $z$ define $\LL^\prime$ in $N$. 
    
    \item $T$ is an $\LL^\prime$-theory in $N$. 
  \end{enumerate}

  \item A cardinal $\kappa$ is an \emph{$F$-$\Psi$-outward compactness cardinal for $\LL$} if all $\LL$-theories that are $F$-$\Psi$-outward satisfiable at $\kappa$ are satisfiable. 
   \end{enumerate}
\end{definition}

In order to motivate the above concept, we now discuss its relation to the compactness properties of $\LL^2$ introduced in Section \ref{section:CharLL2} and show that, for sufficiently strong measures of closeness, these notions coincide. 
 For this purpose, we  make the following definition, which will also be relevant later in this section.

\begin{definition}
 A measure of closeness $\Psi(v_0,\ldots,v_4)$ is \emph{strong} if it is a $\Delta_1^{\ZFC^-}$-formula with $$\ZFC\vdash\forall M,N,\mu,\nu,\rho ~ [\Psi(N,M,\mu,\nu,\rho) ~ \longrightarrow ~ \Psi_{str}(N,M,\mu,\nu,\rho)].$$ 
\end{definition}

Note that, since $F$-$\Psi$--outward compactness is supposed to correspond to $\Psi$-largeness for certain classes $A$ (corresponding to the theory $F$) and these classes may code $V_\lambda$ as a subset of $\lambda$ (and therefore ensure $V_\lambda\subseteq M$ in the context of Definition \ref{definition:PsiLargeA}), strongness can be seen as a natural requirement in the setting of this section.

\begin{proposition}
 Given a strong measure of closeness $\Psi$, an $\LL^2$-theory $T$ is $\Psi$-outward satisfiable at a cardinal $\kappa$ if and only if it is $\ZFC^*$-$\Psi$-outward satisfiable at $\kappa$. In particular, a cardinal $\kappa$ is a $\Psi$-outward compactness cardinal for $\LL^2$ if and only if it is a $\ZFC^*$-$\Psi$-outward compactness cardinal for $\LL^2$. 
\end{proposition}

\begin{proof}
 Our definitions directly ensure that every $\LL^2$-theory that is $\Psi$-outward satisfiable at a cardinal $\kappa$ is also $\ZFC^*$-$\Psi$-outward satisfiable at the given cardinal $\kappa$. 
 In the other direction, assume that a limit ordinal $\eta$ witnesses that an  $\LL^2$-theory $T$ is $\ZFC^*$-$\Psi$-outward satisfiable at a cardinal $\kappa$. Fix a cardinal $\lambda<\kappa$ and a cardinal $\vartheta>\eta$ with $T\in H_\vartheta$. In addition, let $G$ be $\Coll(\omega,\vartheta)$-generic over $V$ and pick an outer $\ZFC^*$-model $N$ of $V_\vartheta^V$ in $V[G]$ with the property that $\Psi(N,V_\vartheta^V,\lambda,\kappa,\eta)$ holds in $V[G]$. 
 We then know that the canonical formulas defining $\LL^2$ define an abstract logic $\LL^\prime$ in $N$ and $T$ is an $\LL^\prime$-theory in $N$. 
 Moreover, these formulas are absolute between $V$ and $V_\lambda$ as well as between $N$ and $V_\lambda^N$. 
 Since our  setup ensures that $V_\lambda^N=V_\lambda^V$, we can now conclude that $T$ is ${<}\lambda$-satisfiable in $N$.  
\end{proof}

As a further motivation for the results proven below, we now derive an analogue of Proposition \ref{proposition:OutwardLessCompact} for the above compactness property.

 \begin{proposition}\label{proposition:VPoutwarCompact}
  Let $\Psi$ be a $\Delta_1^{\ZFC}$-formula that is a measure of closeness and let $n>1$ be a natural number. 
  If  $\LL$ is an abstract logic and $\kappa$  is a limit cardinal, then every $\LL$-theory that is $\ZFC_n$-$\Psi$-outward satisfiable at $\kappa$ is ${<}\kappa$-satisfiable.  
  In particular, if Vop\v{e}nka's Principle holds, then every abstract logic has a $\ZFC_n$-$\Psi$-outward compactness cardinal.  
 \end{proposition}
 
 \begin{proof}
  Fix  formulas  $\varphi_0(v_0,v_1,v_2)$ and $\varphi_1(v_0,v_1,v_2)$ and a parameter $z$ defining an abstract logic $\LL$ and let $\eta$ be a limit ordinal witnessing that an $\LL$-theory $T$ is $\ZFC_n$-$\Psi$-outward satisfiable at a cardinal $\kappa$. 
  Assume, towards a contradiction, that there is a subset $T_0$ of $T$ of cardinality less than $\kappa$ that is an unsatisfiable $\LL$-theory. 
  Pick a cardinal $\lambda<\kappa$ with $\vert T_0\vert<\lambda$ and a cardinal $\vartheta>\eta$ with $T,z\in H_\vartheta=V_\vartheta$ and $V_\vartheta$ sufficiently elementary in $V$. Then \ref{Prop:B} in Definition \ref{definition:Closenessss}.\ref{item:Closeness} implies that $\Psi(V_\vartheta,V_\vartheta,\lambda,\kappa,\eta)$ holds. 
  In addition, our choice of $\vartheta$ ensured that the formulas $\varphi_0$ and $\varphi_1$ are absolute between $V$ and $V_\vartheta$. 
  Finally, we also know that, in $V_\vartheta$,  the formulas $\varphi_0$ and $\varphi_1$ together with the parameter $z$ define an abstract logic $\LL^\prime$, and that $T$ is an $\LL^\prime$-theory  that is not ${<}\lambda$-satisfiable. 
  If $G$ is $\Coll(\omega,\vartheta)$-generic over $V$, then  $V_\vartheta^V$ is an outer $\ZFC_n$-model of   $V_\vartheta^V$ in $V[G]$, and the fact that  $\Psi$ is a $\Delta_1^{\ZFC}$-formula ensures that $\Psi(V_\vartheta^V,V_\vartheta^V,\lambda,\kappa,\eta)$ holds in $V[G]$. 
 However, this contradicts that $T$ is  $\ZFC_n$-$\Psi$-outward satisfiable at $\kappa$ in $V$.     
 
  Finally, assume that Vop\v{e}nka's Principle holds and $\LL$ is an abstract logic. Then Theorem \ref{theorem:Makowsky} shows that there is a limit cardinal $\kappa$ that is a strong compactness cardinal for $\LL$. 
  The above computations now show that for every $n>1$, the cardinal $\kappa$ is a $\ZFC_n$-$\Psi$-outward compactness cardinal for $\LL$. 
 \end{proof}

In the following, we will restrict ourselves to abstract logics with the property that both their formulas and the translations of formulas induced by renamings are produced from the given signatures and  renamings by some simple recursive procedure. 
 It is easy to see that all \emph{finitely generated} logics (see {\cite[\S 1]{MR780522}}) possess the property introduced below. In particular,   we can apply the results below to second-order logic.

\begin{definition}
 An abstract logic $\LL$ with occurrence number $\ooo$ has \emph{simple formulas} if there exists a $\Sigma_1$-formula $\varphi(v_0,\ldots,v_3)$ and a set $z$ with the property that whenever $r$ is a renaming of a language $\sigma$  with less than $\ooo$-many symbols into a language $\tau$ and $\map{r_*}{\LL(\sigma)}{\LL(\tau)}$ is the bijection induced by $r$, then  $$\Set{\langle \phi,r_*(\phi)\rangle}{\phi\in\LL(\sigma)} ~ = ~ \Set{\langle x,y\rangle}{\varphi(r,x,y,z)}$$ holds. 
\end{definition}

Note that if a $\Sigma_1$-formula $\varphi(v_0,\ldots,v_3)$ and a set $z$ witness that an abstract logic $\LL$ with occurrence number $\ooo$ has  simple formulas, then $$\LL(\sigma) ~ = ~ \Set{x}{\varphi(\id_\sigma,x,x,z)}$$ holds for every language $\sigma$ with  less than $\ooo$-many symbols, where $\id_\sigma$ denotes the trivial renaming of $\sigma$ into itself. This shows that the restriction of the class function $\LL$ to languages with   less than $\ooo$-many symbols is also definable by a $\Sigma_1$-formula with parameter $z$ in this case.

While all logics that occur in practice are  simple, it is not hard to construct an example of a non-simple abstract logic. Following the suggestion of the  anonymous referee, we now present such an example.

\begin{example}
 Given a set $y$, we let $\phi_y$ denote the pair $\langle y,\alpha\rangle$, where $\alpha$ is the least ordinal such that  $y\in V_\alpha$ and all $\Sigma_1$-formulas with parameters in $V_\alpha$ are absolute between $V$ and $V_\alpha$ (see {\cite{Cn}}). 
 Then there is a unique abstract logic $\LL$ with occurrence number $2$ and the property that for every language $\tau$, we have $\LL(\tau)=\Set{\phi_r}{r\in \mathfrak{R}_\tau}$ and $$M\models_\LL\phi_r ~ \Longleftrightarrow ~ r^M\neq\emptyset$$ whenever $M$ is a $\tau$-structure and $r\in \mathfrak{R}_\tau$.

 Assume, towards a contradiction, that $\LL$ has simple formulas. Then an earlier observation shows that there is a $\Sigma_1$-formula $\varphi(v_0,v_1,v_2)$ and a set $z$ with the property that for every language $\tau$ with exactly one symbol, the set $\LL(\tau)$ is the unique set $u$ with the property that $\varphi(u,\tau,z)$ holds. 
 Let $\alpha$ be the minimal ordinal such that   $z\in V_\alpha$ and all $\Sigma_1$-formulas with parameters in $V_\alpha$ are absolute between $V$ and $V_\alpha$. In addition, let $\tau$ denote the unique language with $\mathfrak{C}_\tau=\mathfrak{F}_\tau=\emptyset$, $\mathfrak{R}_\tau=\{z\}$ and with arity $a_\tau(z)=1$. Then $\tau\in V_\alpha$ and the set $\LL(\tau)$ witnesses that the $\Sigma_1$-statement $\exists x ~ \varphi(x,\tau,z)$ holds in $V$. Our setup now ensures that this statement holds in $V_\alpha$ and, using the uniqueness of $\LL(\tau)$ and the upwards absoluteness of $\Sigma_1$-formulas, we can conclude that $\LL(\tau)$ is an element of $V_\alpha$. But this yields a contradiction, because $\LL(\tau)=\{\phi_z\}=\{\langle z,\alpha\rangle\}\notin V_\alpha$. 
\end{example}

The next theorem generalizes the forward direction of Theorem \ref{theorem:LL2Duality} to the setting of this section.

\begin{theorem}\label{theorem:ForwardGlobalPrinciples}
 Let   $\Psi$ be a strong measure of closeness. 
 Assume that for every class $A$, there is a proper class of cardinals that are $\Psi$-large for $A$.
  Then for every natural number $n$, every abstract logic with simple formulas has  a $\ZFC_n$-$\Psi$-outward compactness cardinal.  
\end{theorem}

\begin{proof}
 Let $\LL$ be an abstract logic with simple formulas that is defined by formulas  $\varphi_0$ and $\varphi_1$ together with a parameter $z_0$. In the following, let $\ooo$  denote the  occurrence number  of $\LL$. Fix a $\Sigma_1$-formula $\varphi(v_0,\ldots,v_3)$ and a parameter $z_1$ witnessing that $\LL$ has simple formulas.   
  As discussed in Remark \ref{remark:UniquePermutation}, the uniqueness of $\id_{\LL(\tau)}$ with respect to $\id_\tau$ in  \eqref{item:renaming} of Definition \ref{definition:AbstractLogic}  ensures that for every language $\tau$ and every non-trivial  permutation $\pi$ of $\LL(\tau)$, there exists $\chi\in\LL(\tau)$ and a $\tau$-structure $O$ with 
 \begin{equation}\label{equation:ModelsNonTrivial}
  O\models_\LL\chi ~ \Longleftrightarrow ~ O \mathbin{{\not\models}_\LL} \pi(\chi).
 \end{equation}
 This allows us to  fix a cardinal $\rho$ above $\ooo$ that satisfies $z_0,z_1\in H_\rho=V_\rho$ and has the property that for every language $\tau$ in $H_\ooo$ and every non-trivial permutation $\pi$ of $\LL(\tau)$, we have $\LL(\tau)\in H_\rho$ and there exists $\chi\in\LL(\tau)$ and a $\tau$-structure $O$ in $H_\rho$ with \eqref{equation:ModelsNonTrivial}. 
 Define $A$ to be the class $$\{\langle 0,z_0\rangle,\langle 1,z_1\rangle,\langle 2,\rho\rangle\} ~ \cup ~ \Set{\langle 3,y_0,y_1,y_2\rangle}{\varphi_0(y_0,y_1,y_2)} ~ \cup ~ \Set{\langle 4,y_0,y_1,y_2\rangle}{\varphi_1(y_0,y_1,y_2)}$$
 %
  and pick a  cardinal $\kappa>\rho$ that is $\Psi$-large for $A$.  
 Fix a natural number $n>1$. 
  Pick an $\LL$-theory $T$ that is  $\ZFC_n$-$\Psi$-outward satisfiable at $\kappa$, and let $\eta>\kappa$ be a limit ordinal witnessing this.  Let $\sigma$ be  a language with $T\subseteq\LL(\sigma)$. 
 In this situation, we can apply  Lemma \ref{lemma:EmbeddingLargeForA} to find a cardinal $\lambda>\eta$ with $\sigma,\LL(\sigma)\in H_\lambda$, an inner model $M$ and an elementary embedding $\map{j}{V}{M}$ such that $j(\kappa)>\lambda$, $j(A\cap V_\kappa)\cap V_\lambda=A\cap M\cap V_\lambda$ and $\Psi(V_\vartheta,V_\vartheta^M,\lambda,j(\kappa),j(\eta))$ holds for all limit ordinals $\vartheta>j(\eta)$. 
 Our assumptions on $\Psi$ then ensure that $V_\lambda\subseteq M$.
Moreover, by our choice of $A$, we know that $j(\rho)=\rho$, $j(z_0)=z_0$ and $j(z_1)=z_1$. 
  Given $i<2$, we define $$B_i ~ = ~ \Set{\langle y_0,y_1,y_2\rangle\in V_\lambda^3}{\varphi_i(y_0,y_1,y_2)}.$$ Our setup then ensures that $$B_i ~ = ~ j(B_i\cap V_\kappa)\cap V_\lambda ~ = ~ \Set{\langle y_0,y_1,y_2\rangle\in V_\lambda^3}{M\models\varphi_i(y_0,y_1,y_2)}$$ holds for all $i<2$. 
 Finally, we set $\mu=\crit(j)$.

 \begin{claim*}
  $\mu>\rho$ and $z_0,z_1\in V_\mu$. 
 \end{claim*}
 
 \begin{proof}[Proof of the Claim]
  Assume that one of the statements of the claim fails. Then there is a limit  ordinal $\mu<\zeta<\lambda$ with $j(\zeta)=\zeta$. But this implies that the map $\map{j\restriction V_{\zeta+2}}{V_{\zeta+2}}{V_{\zeta+2}}$ is a non-trivial elementary embedding, contradicting the \emph{Kunen Inconsistency}. 
 \end{proof}

 Next, note that $j(\sigma)$ is a language in $V$.

  \begin{claim*}
  $j[T]\subseteq\LL(j(\sigma))$. 
 \end{claim*}
 
 \begin{proof}[Proof of the Claim]
  Fix $\phi\in T$. We can find a sublanguage $\bar{\sigma}$ of $\sigma$ with less than $\ooo$-many symbols that satisfies $\phi\in\LL(\bar{\sigma})$. By our assumptions, we know that $\varphi(\id_{\bar{\sigma}},\phi,\phi,z_1)$ holds in $V$ and hence elementarity implies that $\varphi(\id_{j(\bar{\sigma})},j(\phi),j(\phi),z_1)$ holds in $M$. 
 But then $\Sigma_1$-upwards absoluteness ensures that this statement also holds in $V$. Since our first claim shows that $j(\bar{\sigma})$ is a sublanguage of $j(\sigma)$ with less than $\ooo$-many symbols, $j(\phi)\in\LL(j(\bar{\sigma}))\subseteq\LL(j(\sigma))$.  
 \end{proof}

In particular, this shows that $j[T]$ is an $\LL$-theory in $V$. 
 Next, notice that  the elementarity of $j$  ensures that, in $M$, the formulas $\varphi_0$ and $\varphi_1$ together with the parameter $z_0$ define an abstract logic $\LL_M$ and the limit ordinal $j(\eta)$ witnesses that the $\LL_M$-theory $j(T)$ is  $\ZFC_n$-$\Psi$-outward satisfiable at $j(\kappa)$.

\begin{claim*}
 The $\LL$-theory $j[T]$ is satisfiable in $V$. 
\end{claim*}

\begin{proof}[Proof of the Claim]
   Let  $\vartheta>j(\lambda)$ be a   cardinal such that   $j(\vartheta)=\vartheta$,  the set $V_\vartheta$ is sufficiently   elementary in $V$, and the set $V_\vartheta^M$ is sufficiently elementary in $M$.  Then, $\Psi(V_\vartheta,V_\vartheta^M,\lambda,j(\kappa),j(\eta))$ holds.  
  Let $G$ be $\Coll(\omega,\vartheta)$-generic over $V$. 
  Since $\Coll(\omega,\vartheta)^V=\Coll(\omega,\vartheta)^M\in M$, it follows that $G$ is also $\Coll(\omega,\vartheta)$-generic over $M$. 
  Then, the set  $V_\vartheta^M$ is countable in $M[G]$ and we know that, in $M[G]$, 
  $j(T)$ is a ${<}\lambda$-satisfiable $\LL^\prime$-theory in $N$ 
  whenever $N$ is a countable outer $\ZFC_n$-model of $V_\vartheta^M$, $\LL^\prime$ is an abstract logic in $N$, and the following statements hold: 
  \begin{enumerate}[label=(\alph*)]
   \item $\Psi(N,V_\vartheta^M,\lambda,j(\kappa),j(\eta))$ holds in $M[G]$. 
   
   \item If $i<2$ and $\langle y_0,y_1,y_2\rangle\in B_i$, then $\varphi_i(y_0,y_1,y_2)$ holds in $N$.  
   
   \item The formulas $\varphi_0$ and $\varphi_1$ together with the parameter $z_0$ define  $\LL^\prime$ in $N$. 
   
   \item  $j(T)$ is an  $\LL^\prime$-theory in $N$. 
  \end{enumerate}
  Since the given statement can be expressed by a $\Pi_1$-formula with parameters in $H_{\aleph_1}^{M[G]}$, \emph{Shoenfield Absoluteness} ensures that it also holds in $V[G]$. 
 In addition, $\Sigma_1$-upwards absoluteness  implies that $\Psi(V_\vartheta^V,V_\vartheta^M,\lambda,j(\kappa),j(\eta))$ holds in $V[G]$. 
Moreover, since $V_\vartheta^V$ was chosen sufficiently elementary in $V$, we know that for all $i<2$ and all $\langle y_0,y_1,y_2\rangle\in B_i$, the statement $\varphi_i(y_0,y_1,y_2)$ holds in $V_\vartheta^V$. 
 Our choice of $\vartheta$ also ensures that the formulas $\varphi_0$ and $\varphi_1$ together with the parameter $z_0$ define an abstract logic $\LL^\prime$ in $V_\vartheta^V$ and $j(T)$ is an  $\LL^\prime$-theory in $V_\vartheta^V$. 
This allows us to conclude that $j(T)$ is a ${<}\lambda$-satisfiable $\LL^\prime$-theory in $V_\vartheta^V$. Since $j[T]\subseteq j(T)$ is an element of $V_\vartheta^V$ and this set has cardinality less than $\lambda$ in $V_\vartheta^V$, we now know that $j[T]$ is a satisfiable $\LL^\prime$-theory in $V_\vartheta^V$. In this situation, the fact that  $V_\vartheta^V$ was chosen to be sufficiently elementary in $V$ allows us to conclude that $j[T]$ is a satisfiable $\LL$-theory in $V$.  
\end{proof}

Let $\sigma_j$ denote the sublanguage of $j(\sigma)$ that is given by the pointwise image of the symbols in $\sigma$ under $j$, and let $r$ denote the canonical renaming from $\sigma$ into $\sigma_j$ induced by $j$.  
 Since the occurrence number $\ooo$ of $\LL$ is smaller than the critical point $\mu$ of $j$, we know that $j[T]\subseteq\LL(\sigma_j)$. 
 Moreover, the previous claim shows that there is a $\sigma$-structure $N$ with the property that the $\sigma_j$-structure $r^*(N)$ obtained from $N$ using the renaming $r$ is a model of $j[T]$. 
 Assume, towards a contradiction, that $N$ is not a model of $T$ and fix $\phi$ in $T$ such that $N\mathbin{{\not\models}_\LL}\phi$. 
  Pick a sublanguage $\bar{\sigma}$ of $\sigma$ that contains less than $\ooo$-many symbols and has the property that $\phi$ is an element of $\LL(\bar{\sigma})$. Then $j$ induces a renaming $\bar{r}$ of $\bar{\sigma}$ into the sublanguage $j(\bar{\sigma})$ of $\sigma_j$, and this function is equal to the restriction of $r$ to $\bar{\sigma}$. 
  Let $\bar{N}$ denote the $\bar{\sigma}$-reduct of $N$. Then $\bar{N}\mathbin{{\not\models}_\LL}\phi$. Moreover,  the $j(\bar{\sigma})$-structure  $\bar{r}^*(\bar{N})$ obtained from $\bar{N}$ using $\bar{r}$ is equal to the $j(\bar{\sigma})$-reduct of $r^*(N)$ and therefore we know that $\bar{r}^*(\bar{N})\models_\LL j(\phi)$  and $\bar{r}^*(\bar{N})\mathbin{{\not\models}_\LL} \bar{r}_*(\phi)$.  In particular, it follows that $\bar{r}_*(\phi)\neq j(\phi)$.
  
  Now, pick a renaming $s$ of a language $\tau\in H_\ooo$ into $\bar{\sigma}$. Then $j(s)$ is a renaming of  $\tau$ into $j(\bar{\sigma})$ and our setup ensures that $j(s)=\bar{r} \circ s$, {i.e.,} we know  that the following diagram commutes:

 \begin{equation*}
\begin{xy}
\xymatrix{
 \bar{\sigma} \ar[rr]^{\bar{r}} & & j(\bar{\sigma}) \\
    & &   \\
   \tau \ar[uu]^{s}  \ar@{-}@<0.2ex>[rr]^{\id_\tau} \ar@{-}@<-0.2ex>[rr]  & &  \tau \ar[uu]_{j(s)} 
}
\end{xy}
\end{equation*}

Since $V_\lambda\subseteq M$, we know that the sets $\bar{\sigma}$, $j(\bar{\sigma})$ and $\tau$ as well as the maps $s$ and $j(s)$ are all elements of $M$, and this implies that the renaming $\bar{r}$ is also contained in $M$. 
 %
 %
 In addition, 
  our setup ensures that $\LL(\bar{\sigma})=\LL_M(\bar{\sigma})$, $\LL(j(\bar{\sigma}))=\LL_M(j(\bar{\sigma}))$ and $\LL(\tau)=\LL_M(\tau)$. 
  Let  $\map{s_*}{\LL(\tau)}{\LL(\bar{\sigma})}$ denote the bijection induced by $s$ in $V$. 
  Since $\LL$ has simple formulas and $j(z_1)=z_1$, elementarity and $\Sigma_1$-upwards absoluteness imply 
  $s_*$ is an element of  $M$ and  this  function is also the unique bijection between $\LL_M(\tau)$ and $\LL_M(\bar{\sigma})$ that is induced by $s$ in $M$. 
   Since $j(s_*)$ is the bijection between $\LL_M(\tau)$ and $\LL_M(j(\bar{\sigma}))$ induced by $j(s)$ in $M$, we know that $j(s_*)\circ s_*^{{-}1}$ is the bijection between $\LL_M(\bar{\sigma})$ and $\LL_M(j(\bar{\sigma}))$ induced by ${\bar{r}}$ in $M$.

  \begin{claim*}
   $j(s_*)\circ s_*^{{-}1}=j\restriction\LL(\bar{\sigma})$. 
  \end{claim*}
  
 \begin{proof}[Proof of the Claim]
  Assume, towards a contradiction, that the statement of the claim fails. Since  $j\restriction\LL(\bar{\sigma})$ is a bijection between $\LL(\bar{\sigma})$ and $\LL(j(\bar{\sigma}))$, we then know that the map $j(s_*)^{{-}1}\circ j \circ s_*$ is a non-trivial permutation of $\LL(\tau)$. 
 In this situation, our setup yields  $\chi\in\LL(\tau)$ and a $\tau$-structure $O\in H_\rho$ with  $$O\models_\LL 
 \chi ~ \Longleftrightarrow ~ O\mathbin{{\not\models}_\LL}(j(s_*)^{{-}1}\circ j \circ s_*)(\chi).$$  
  Note that $\LL(\tau)=\LL_M(\tau)$ and, since $\LL(\tau)\subseteq V_\lambda$, our setup ensures that $$O\models_\LL \upsilon ~ \Longleftrightarrow ~ O\models_{\LL_M}\upsilon$$ holds for every $\upsilon\in\LL(\tau)$. 
  Moreover, if $s^*(O)$ denotes the $\bar{\sigma}$-structure obtained from $O$ using $s$, then, in $M$, the $j(\bar{\sigma})$-structure $j(s^*(O))$ is obtained from $O$ using the renaming $j(s)$. 
  Since $j(s_*)$ is the  bijection between $\LL_M(\tau)$ and $\LL_M(j(\bar{\sigma}))$ induced by $j(s)$ in $M$, we then know that $$O\models_{\LL_M}j(s_*)(\upsilon) ~ \Longleftrightarrow ~ j(s^*(O))\models_{\LL_M}\upsilon$$ holds for all $\upsilon\in\LL_M(j(\bar{\sigma}))$. 
  By combining these equivalences with the elementarity of $j$, we can now conclude that  
  \begin{equation*}
   \begin{split}
    O\models_\LL\chi ~ & \Longleftrightarrow ~ s^*(O) \models_\LL s_*(\chi) ~ \Longleftrightarrow ~ j(s^*(O)) \models_{\LL_M} j(s_*(\chi)) \\ 
       & \Longleftrightarrow ~ O\models_{\LL_M}(j(s_*)^{{-}1}\circ j\circ s_*)(\chi) ~ \Longleftrightarrow ~ O\models_\LL(j(s_*)^{{-}1}\circ j\circ s_*)(\chi), 
    \end{split}
  \end{equation*}
 a contradiction.  
 \end{proof}

 The above claim shows that, in $M$, the map $j\restriction\LL(\bar{\sigma})$ is the bijection between $\LL_M(\bar{\sigma})$ and $\LL_M(j(\bar{\sigma}))$ induced by $\bar{r}$. This directly implies that $\varphi(\bar{r},\phi,j(\phi),z_1)$ holds in $M$. By $\Sigma_1$-upwards absoluteness, we know that this statement holds in $V$ and we can conclude that $\bar{r}_*(\phi)=j(\phi)$, a contradiction. 
\end{proof}

 The following result provides a generalization of the backward direction of Theorem \ref{theorem:LL2Duality} that will turn out to be suitable for our characterizations of fragments of Vop\v{e}nka's Principle.

\begin{theorem}\label{theorem:SchemesBackward}
   Let   $\Psi$ be an  $\LL^2$-measure of closeness. 
   Assume that for every natural number $n$, every abstract logic that extends second-order logic and has simple formulas has a $\ZFC_n$-$\Psi$-outward compactness cardinal.  Then, for every formula $\varphi(v)$ in the language of set theory,  there exists a cardinal that is  $\Psi$-strong for the class $\Set{x}{\varphi(x)}$.   
\end{theorem}

\begin{proof}
 Let $A$ denote the class $\Set{x}{\varphi(x)}$.   
 We define an abstract logic $\LL_\varphi$ that extends second-order logic by adding atomic formulas $\phi_E(v_0,v_1)$ for every binary relation  symbol $E$ and all variables $v_0$ and $v_1$,  and defining $M\models_{\LL_\varphi}\phi_E(c_0,c_1)$ to hold for constant symbols $c_0$ and $c_1$, whenever $E^M$ is a well-founded and extensional relation on $\vert M\vert$ and the corresponding transitive collapse $\map{\pi}{\vert M\vert}{N}$ sends $c_0^M$ to an ordinal $\lambda$ and $c_1^M$ to a set $X$ with the property that $$V_\lambda\cap X ~ = ~ \Set{x\in N\cap V_\lambda}{\varphi(x)}.$$ 
  Let  $\varphi_0(v_0,v_1)$ and $\varphi_1(v_0,v_1)$ denote the canonical formulas defining $\LL_\varphi$ (without using additional parameters), and pick a natural number $n>1$ such that $\ZFC_n$ proves that these formulas define the  abstract logic $\LL_\varphi$. 
 Finally, since it is immediate that  $\LL_\varphi$ has simple formulas, our assumptions ensure that there is a $\ZFC_n$-$\Psi$-outward compactness cardinal $\kappa$ for $\LL_\varphi$. 

 Now,  let $\theta>\eta>\kappa$ be limit ordinals and fix a cardinal $\zeta>\eta$.
  Let $\Psi^*(v_0,v_1,v_2)$ be an $\LL^2$-formula in the language of set theory witnessing that $\Psi$ is an  $\LL^2$-measure of closeness, and let $\Psi^\prime(v_0,\ldots,v_3)$ denote the relativisation of $\Psi^*(v_0,v_1,v_2)$ to $v_3$, {i.e.,} $\Psi^\prime$ has the property that $\ZFC^*$ proves that $\langle M,\in\rangle\models_{\LL^2}\Psi^\prime(a,b,c,d)$ is equivalent to $\langle d,\in\rangle\models_{\LL^2}\Psi^*(a,b,c)$ whenever $M$ is a  transitive set, $d\in M$ is transitive and $a,b,c\in d$. 
    Consider the language that extends the language of set theory by a constant symbol  $b$, constant symbols $c_x$ for all elements $x$ of $V_{\theta+1}$ and   constant symbols $d_\gamma$ for all $\gamma\leq\zeta$.  Let $T$ denote the $\LL_\varphi$-theory consisting of the following:
  \begin{enumerate}
    \item The first-order elementary diagram of $V_{\theta+1}$, using the constant symbols $c_x$ for $x$ in $V_{\theta+1}$. 
    
    \item The sentence $\anf{d_\zeta<b<c_\kappa}$ and all sentences of the form $\anf{d_\beta<d_\gamma<c_\kappa}$ for $\beta<\gamma\leq\zeta$. 
 
    \item The sentence $\Psi^\prime(b,c_\kappa,c_\eta,c_{V_\theta})$.  
    
    \item The sentence $\phi_\in(b,c_{A\cap V_\kappa})$. 
  \end{enumerate}

 \begin{claim*}
  The ordinal $\eta$ witnesses that $T$ is $\ZFC_n$-$\Psi$-outward satisfiable at $\kappa$.
 \end{claim*}
 
 \begin{proof}[Proof of the Claim]
  Fix a cardinal $\lambda<\kappa$ and a cardinal $\vartheta>\eta$ with $T\in H_\vartheta$. 
   We then know that $\vartheta>\theta$. 
   Let $G$ be $\Coll(\omega,\vartheta)$-generic over $V$, and let $N$ be an outer $\ZFC_n$-model    of $V_\vartheta^V$ in $V[G]$ such that $\Psi(N,V_\vartheta^V,\lambda,\kappa,\eta)$ holds in $V[G]$, and  the formulas $\varphi_0$ and $\varphi_1$ are absolute from $V$ to $N$ with respect to parameters in $V_\lambda^V$. We then know that $T$ is an $\LL_\varphi$-theory in $N$. 
 Next,  we apply  \ref{Prop:F} in Definition \ref{definition:Closenessss}.\ref{item:Closeness} to show that $\Psi(V_\theta^N,V_\theta^V,\lambda,\kappa,\eta)$ holds in $V[G]$ and, using the fact that $\Psi$ is a $\Delta_1^{\ZFC^*}$-formula, we can conclude that this statement also holds in $N$. 
 We then know that 
  $$\langle V_{\theta+1}^V,\in\rangle\models_{\LL_\varphi}\Psi^\prime(\lambda,\kappa,\eta,V_\theta^V)$$ holds in $N$.

  For all $\alpha<\lambda$, our setup ensures that $$\langle V_{\alpha+1},\in\rangle\models_{\LL_\varphi}\phi_\in(\alpha,A\cap V_\alpha)$$ holds in $V$ and, since all parameters appearing in this statement are elements of $V_\lambda$, our assumptions on $N$ ensure that it also holds in $N$. 
   Hence, we know that $A\cap V_\lambda^N$ is equal to the set of all $x\in V_\lambda^N\cap V_\lambda^V$ with the property that $\varphi(x)$ holds in $N$, and this allows us to conclude that $$\langle V^V_{\theta+1},\in\rangle\models_{\LL_\varphi}\phi_\in(\lambda,A\cap V_\kappa)$$ holds in $N$.

   Fix a subtheory $T_0$   of $T$ in $N$ that has cardinality less than $\lambda$ in $N$.  The fact that  $\lambda$ is a cardinal in $N$ now allows us to construct a model of $T_0$ with domain $V_{\theta+1}^V$ in $N$ that interprets  $b$ as $\lambda$    and all constant symbols of the form $d_\gamma$ that appear in sentences in $T_0$ as ordinals less than $\lambda$.  
 \end{proof}

  Since $\kappa$ is a   $\ZFC_n$-$\Psi$-outward compactness cardinal for $\LL_\varphi$, the above claim shows that the theory $T$ is satisfiable. 
  This allows us to find a transitive set $M$, an ordinal $\zeta<\lambda\in M$ and an elementary embedding $\map{j}{V_{\theta+1}}{M}$ with $j(\kappa)>\lambda$ and $$\langle M,\in\rangle\models_{\LL_\varphi}\phi_\in(\lambda,j(A\cap V_\kappa)) ~ \wedge ~ \Psi^\prime(\lambda,j(\kappa),j(\eta),j(V_\theta)).$$ 
  We then know that $$j(A\cap V_\kappa)\cap V_\lambda ~ = ~ A\cap M\cap V_\lambda.$$
  Moreover, since elementarity implies that $j(V_\theta)=M\cap V_{j(\theta)}$, we can conclude that $$\langle M\cap V_\lambda,\in\rangle\models_{\LL_\varphi}\Psi^*(\lambda,j(\kappa),j(\eta))$$ and this shows that $\Psi(V_{j(\theta)},M\cap V_{j(\theta)},\lambda,j(\kappa),j(\eta))$ holds. 
  These computations show that the cardinal $\kappa$ is $\Psi$-strong for the class $A$.  
\end{proof}

We can now combine the above results to characterize the principles discussed in Section \ref{section:FragmentsVP} through compactness properties of abstract logics. 

\begin{corollary}\label{corollary:OutwardVP}
The following schemes are equivalent over $\ZFC$: 
   \begin{enumerate}
    \item\label{item::VP1} Vop\v{e}nka's Principle. 
    
    \item\label{item::VP2} For every natural number $n$ and every abstract logic $\LL$, there exists a $\ZFC_n$-$\Psi_{ext}$-outward compactness cardinal for $\LL$. 
            
    \item\label{item::VP3} For every natural number $n$ and every abstract logic $\LL$ with simple formulas, there exists a $\ZFC_n$-$\Psi_{sc}$-outward compactness cardinal for $\LL$. 
   \end{enumerate}
  \end{corollary}
  
  \begin{proof}
   First, assume that \eqref{item::VP3} holds. Then Theorem \ref{theorem:SchemesBackward} implies that for every formula $\varphi(v)$ in the language of set theory, there exists a cardinal that is $\Psi_{sc}$-large for the class $\Set{x}{\varphi(x)}$. An application of Lemma \ref{lemma:LargeVP} then shows that \eqref{item::VP1} holds. This completes the proof of the corollary, because the implication from \eqref{item::VP1} to \eqref{item::VP2} is already given by Proposition \ref{proposition:VPoutwarCompact} and the implication from  \eqref{item::VP2} to  \eqref{item::VP3} holds trivially.  
  \end{proof}

  \begin{corollary}\label{corollary:OutwardOrdWoodin}
  The following schemes are equivalent over $\ZFC$: 
    \begin{enumerate}
    \item\label{item::OrdWoodin1} $\Ord$ is Woodin. 
    
    \item\label{item::OrdWoodin2} For every natural number $n$ and every abstract logic $\LL$ with simple formulas, there exists a $\ZFC_n$-$\Psi_{str}$-outward compactness cardinal for $\LL$. 
   \end{enumerate}
\end{corollary}

\begin{proof}
 If we assume that \eqref{item::OrdWoodin2} holds, then Theorem \ref{theorem:SchemesBackward} implies that for every formula $\varphi(v)$ in the language of set theory, there exists a cardinal that is $\Psi_{str}$-large for the class $\Set{x}{\varphi(x)}$, and then Lemma \ref{lemma:OrdWoodinEqui} shows that \eqref{item::OrdWoodin1} holds. 
 In the other direction, assume that \eqref{item::OrdWoodin1} holds. Then Lemma \ref{lemma:OrdWoodinEqui}  shows that for every class $A$, there is a proper class of cardinals that are $\Psi_{str}$-large for $A$. In this situation, Theorem \ref{theorem:ForwardGlobalPrinciples} shows that \eqref{item::OrdWoodin2} holds. 
\end{proof}


\section{Open questions}

We close this paper by listing some questions that are raised by the results of this paper. 
First, as already mentioned in Section \ref{subsection:Omega1StronglyCompact}, results of Bagaria and Magidor in \cite{MR3152715} and \cite{MR3226024} entail  that the combinatorics of $\Psi_{stc}$-large cardinals are fundamentally different from those of the other large cardinal notions studied in Section \ref{section:PsiLarge}. For example, it is consistent that the least $\Psi_{stc}$-large cardinal is singular and therefore  all embeddings witnessing the $\Psi_{stc}$-largeness of the given cardinal have critical points that are strictly smaller than this cardinal. It is therefore natural to ask if such characteristics of the given large cardinal notion can be read off from the formula inducing it:

\begin{question}
 Is there a natural characterization of the collection of all formulas $\Psi$ with the property that it is provable that if there is a $\Psi$-large cardinal, then the least such cardinal  is regular?  
\end{question}

 In a similar direction, the above results also raise the question whether they are somehow connected to the \emph{Identity Crisis} phenomenon, first studied by Magidor in \cite{MR0429566}. More precisely, a combination of Corollary \ref{corollary:IndividualChar}(\ref{item:CorollaryCharOmega1SC}) with results in \cite{MR3152715} and \cite{MR0429566} shows that, assuming the consistency of sufficiently strong large cardinal axioms, the axioms of $\ZFC$ do not decide whether the first $\Psi_{stc}$-large cardinal is the first measurable cardinal.  In contrast,  Corollary \ref{corollary:IndividualChar} also shows that if $\Psi_{ext}$-, $\Psi_{sc}$- or $\Psi_{str}$-large cardinals exist, then the first cardinal of this type is bigger than the first measurable cardinal.

 \begin{question}
   Is there a natural characterization of the collection of all formulas $\Psi$ with the property that  the existence of a $\Psi$-large cardinals provably implies that the least cardinal of this type is bigger than the first measurable cardinal?
 \end{question}

 Next, we consider the question whether other canonical formulas induce large cardinal notions that also consistently exhibit unusual behavior. 
 Remember that a cardinal $\kappa$ is \emph{globally superstrong} (see \cite{MR3904884}) if for every $\lambda>\kappa$, there exists a transitive class $M$ and an elementary embedding $\map{j}{V}{M}$ with $\crit(j)=\kappa$, $j(\kappa)>\lambda$ and $V_{j(\kappa)}\subseteq M$. 
There is an obvious candidate for a formula $\Psi$ such that $\Psi$-largeness corresponds to global superstrongness: let $\Psi_{gss}(v_0,\ldots,v_4)$ denote the canonical formula in the language of set theory with the property that $\Psi_{gss}(N,M,\mu,\nu,\rho)$ holds if and only if the tuple $\langle N,M,\mu,\nu,\rho\rangle$ is suitable   and $N\cap V_\nu\subseteq M$. Then, $\Psi_{gss}$ is an $\LL^2$-measure of closeness. 
  Moreover, it is easy to see that every globally superstrong cardinal is $\Psi_{gss}$-large. But, it is not clear if it is possible to prove a variation of Lemma \ref{lemma:CharStrong} for this formula.

  \begin{question}
   Is it provable that a cardinal $\kappa$ is $\Psi_{gss}$-large if and only if there is a globally superstrong cardinal less than or equal to $\kappa$?
   Is it consistent that the least $\Psi_{gss}$-large cardinal is singular? 
  \end{question}

  Finally, motivated by the observations made in Section \ref{section:Natural?}, we ask if there is a canonical outward compactness characterization for strong compactness:

\begin{question}
 Is there an $\LL^2$-measure of closeness $\Psi$  that naturally induces a large cardinal property below extendibility  and has the property that $\ZFC$ proves that a cardinal $\kappa$ is $\Psi$-large if and only if there is a strongly compact cardinal less than or equal to $\kappa$?  
\end{question}


 \bibliographystyle{amsplain}
\bibliography{references}

\providecommand{\bysame}{\leavevmode\hbox to3em{\hrulefill}\thinspace}
\providecommand{\MR}{\relax\ifhmode\unskip\space\fi MR }
\providecommand{\MRhref}[2]{%
  \href{http://www.ams.org/mathscinet-getitem?mr=#1}{#2}
}
\providecommand{\href}[2]{#2}
\begin{thebibliography}{10}

\bibitem{Cn}
Joan Bagaria, \emph{{$C^{(n)}$}-cardinals}, Archive for Mathematical Logic
  \textbf{51} (2012), no.~3-4, 213--240.

\bibitem{MR3323199}
Joan Bagaria, Carles Casacuberta, A.~R.~D. Mathias, and Ji\v{r}\'{\i}
  Rosick\'{y}, \emph{Definable orthogonality classes in accessible categories
  are small}, J. Eur. Math. Soc. (JEMS) \textbf{17} (2015), no.~3, 549--589.

\bibitem{MR3152715}
Joan Bagaria and Menachem Magidor, \emph{Group radicals and strongly compact
  cardinals}, Trans. Amer. Math. Soc. \textbf{366} (2014), no.~4, 1857--1877.

\bibitem{MR3226024}
\bysame, \emph{On {$\omega_1$}-strongly compact cardinals}, J. Symb. Log.
  \textbf{79} (2014), no.~1, 266--278.

\bibitem{bagaria_wilson_2022}
Joan Bagaria and Trevor~M. Wilson, \emph{The weak {V}op\v{e}nka principle for
  definable classes of structures}, J. Symb. Log. \textbf{88} (2023), no.~1,
  145--168.

\bibitem{barwise}
Jon Barwise, \emph{Admissible sets and structures: An approach to definability
  theory}, Perspect. Logic, vol.~7, 1975.

\bibitem{MR4093885}
Will Boney, \emph{Model theoretic characterizations of large cardinals}, Israel
  J. Math. \textbf{236} (2020), no.~1, 133--181.

\bibitem{bdgm}
Will Boney, Stamatis Dimopoulos, Victoria Gitman, and Menachem Magidor,
  \emph{Model theoretic characterizations of large cardinals revisited}, Trans.
  Amer. Math. Soc. \textbf{377} (2024), no.~10, 6827--6861.

\bibitem{MR3904884}
Jinglun Cai and Konstantinos Tsaprounis, \emph{On strengthenings of superstrong
  cardinals}, New York J. Math. \textbf{25} (2019), 174--194.

\bibitem{MR409165}
C.~C. Chang and H.~J. Keisler, \emph{Model theory}, Studies in Logic and the
  Foundations of Mathematics, vol. Vol. 73, North-Holland Publishing Co.,
  Amsterdam-London; American Elsevier Publishing Co., Inc., New York, 1973.

\bibitem{MR926461}
Katsuya Eda and Yoshihiro Abe, \emph{Compact cardinals and abelian groups},
  Tsukuba J. Math. \textbf{11} (1987), no.~2, 353--360.

\bibitem{FuSa22}
Saka\'e{} Fuchino and Hiroshi Sakai, \emph{{Weakly extendible cardinals and
  compactness of extended logics}}, preprint, 2022.

\bibitem{MR1221741}
Wilfrid Hodges, \emph{Model theory}, Encyclopedia of Mathematics and its
  Applications, vol.~42, Cambridge University Press, Cambridge, 1993.

\bibitem{MR1940513}
Thomas Jech, \emph{Set theory}, Springer Monographs in Mathematics,
  Springer-Verlag, Berlin, 2003, The third millennium edition, revised and
  expanded.

\bibitem{kanamori}
Akihiro Kanamori, \emph{The higher infinite}, second ed., Springer Monographs
  in Mathematics, Springer-Verlag, Berlin, 2003, Large cardinals in set theory
  from their beginnings.

\bibitem{SubtleOrd}
Philipp L{\"{u}}cke, \emph{Weak compactness cardinals for strong logics and
  subtlety properties of the class of ordinals}, preprint, 2024.

\bibitem{MR0295904}
Menachem Magidor, \emph{On the role of supercompact and extendible cardinals in
  logic}, Israel Journal of Mathematics \textbf{10} (1971), 147--157.

\bibitem{MR0429566}
\bysame, \emph{How large is the first strongly compact cardinal? or {A} study
  on identity crises}, Ann. Math. Logic \textbf{10} (1976), no.~1, 33--57.

\bibitem{MR780522}
J.~A. Makowsky, \emph{Vop\v{e}nka's principle and compact logics}, J. Symbolic
  Logic \textbf{50} (1985), no.~1, 42--48.

\bibitem{MR3151400}
Konstantinos Tsaprounis, \emph{Elementary chains and {$C^{(n)}$}-cardinals},
  Arch. Math. Logic \textbf{53} (2014), no.~1-2, 89--118.

\bibitem{MR4439579}
Trevor~M. Wilson, \emph{The large cardinal strength of weak {V}openka's
  principle}, J. Math. Log. \textbf{22} (2022), no.~1, Paper No. 2150024, 15.

\end{thebibliography}


\end{document}